\documentclass{amsart}

\usepackage{amssymb}
\usepackage{amsmath}
\usepackage{amsfonts}
\usepackage{geometry}
\usepackage{hyperref}
\usepackage{graphicx}
\usepackage{amscd}
\usepackage{amsthm}
\usepackage{mdef}
\usepackage{hyphenat}

\usepackage{color}

\newcommand{\barY}{\overline{Y}}

\newtheorem{theorem}{Theorem}

\newtheorem{condition}[theorem]{Condition}

\newtheorem{corollary}[theorem]{Corollary}

\newtheorem{definition}[theorem]{Definition}

\newtheorem{lemma}[theorem]{Lemma}

\newtheorem{remark}[theorem]{Remark}

\begin{document}

\author{Christian Bayer}
\address{Christian Bayer \\ Weierstrass Institute Berlin, Germany}
\email{Christian.Bayer@wias-berlin.de}

\author{Peter K.\ Friz}
\address{Peter K.\ Friz \\ Technische Universit\"at Berlin and Weierstrass Institute Berlin, Germany}
\email{friz@math.tu-berlin.de}

\author{Sebastian Riedel}
\address{Sebastian Riedel \\ Technische Universit\"at Berlin, Germany}
\email{riedel@math.tu-berlin.de}

\author{John Schoenmakers}
\address{John Schoenmakers \\ Weierstrass Institute Berlin, Germany}
\email{schoenma@wias-berlin.de}

\keywords{Gaussian processes, stochastic differential equations, numerical methods for stochastic equations, Monte Carlo methods}

\subjclass[2010]{60G15, 60H10, 60H35, 65C05, 65C30}

\title{From rough path estimates to multilevel Monte Carlo}

\begin{abstract}
  New classes of stochastic differential equations can now be studied using 
  rough path theory (e.g. Lyons et al. \cite{LCL07} or Friz--Hairer \cite{FH14}).  In this paper we investigate,
  from a numerical analysis point of view, stochastic differential equations
  driven by Gaussian noise in the aforementioned sense. Our focus lies on
  numerical implementations, and more specifically on the saving possible via
  multilevel methods. Our analysis relies on a subtle combination of pathwise
  estimates, Gaussian concentration, and multilevel ideas.  Numerical examples
  are given which both illustrate and confirm our findings.
\end{abstract}

% \tableofcontents

\maketitle

\section{Introduction}

We consider implementable schemes for large classes of stochastic
differential equations (SDEs) 
\begin{equation}
dY_{t}=V_{0}\left( Y_{t}\right) dt+\sum_{i=1}^{d}V_{i}\left( Y_{t}\right)
\,dX_{t}^{i}\left( \omega \right)   \label{SDEIntro}
\end{equation}
driven by multidimensional Gaussian signals, say $X=X_{t}\left( \omega \right)
\in \mathbb{R}^{d}.$ The interpretation of these equations is in Lyons'
\textit{rough path} sense \cite{LQ02, LCL07, FV10, FH14}% , an important special
% case of Hairer's celebrated theory of regularity structures \cite{Hai14},
% cf.~\eqref{RDEIntro}
.  In essence, one deals with generalizations of classical
Stratonovich SDE meaning for such equations.  One requires some
smoothness/boundedness conditions on the vector fields $V_{0}$ and $V\equiv
\left( V_{1},\dots ,V_{d}\right) $; for the sake of this introduction, the
reader may assume bounded vector fields with bounded derivatives of all order
(but we will be more specific later). This also requires a ``natural'' lift of
$X\left( \cdot ,\omega \right) $ to a (random) rough path
\begin{equation}
  \label{eq:enhanced-path}
  \mathbf{X}_{t}\left( \omega \right) = \sum_{i=1}^N \int_{0<s_1 <
    \cdots < s_i < t} dX_{s_1}(\omega) \otimes \cdots \otimes
  dX_{s_i}(\omega), 
\end{equation}
see, e.g., \cite[Ch.~3]{LCL07}. In~\eqref{eq:enhanced-path}, $N$ is
related to the roughness of $X$. For instance, in the case of Brownian motion,
we need $N = 2$, but rougher processes $X$ require $N > 2$.
The reader not familiar with rough path theory may think
of $Y$ as ``Stratonovich'' solution to (\ref{SDEIntro}). In fact, $Y$ is known
to be the Wong-Zakai limit, obtained by replacing $X$ in (\ref{SDEIntro})
by piecewise-linear approximation followed by taking the mesh-to-zero limit.

We shall simplify the discussion by choosing $V_{0}\equiv 0$  and using the
short-hand notation
\begin{equation}
dY_t = V\left( Y_t \right) d\mathbf{X}_t,  \label{RDEIntro}
\end{equation}
indicating that the differential equation is driven by the rough path
$\mathbf{X}$ given in~\eqref{eq:enhanced-path}.  Of course, it would be easy
to include equations of the form~\eqref{SDEIntro} into the
framework~\eqref{RDEIntro}, e.g., by including time $t$ as an additional
(smooth) component of the noise $X$. This setting includes, for instance,
fractional Brownian motion (fBm) with Hurst parameter $H>1/4$, see
\cite{CQ02}. It may help the reader to recall that, in the case when $X=B$, a
multidimensional Brownian motion, all this amounts to enhance $B$ with
\textit{L\'{e}vy's stochastic area} or, equivalently, with all iterated
stochastic integrals of $B$ against itself, say $\mathbb{B}%
_{s,t}=\int_{s}^{t}B_{s,u}\otimes dB_u$. The (rough-)pathwise solution concept
then agrees with the usual notion of an SDE solution (in It\^{o}- or
Stratonovich sense, depending on which integration was used in defining $%
\mathbb{B}$). As is well-known this provides a robust extension of the usual
It\^{o} framework of stochastic differential equations with an exploding
number of new applications (including non-linear SPDE theory, robustness of
the filtering problem, non-Markovian H\"{o}rmander theory).

Gaussian noise more general than Brownian motion is also interesting from a
modelling point of view. Many stochastic models are based on independent
identically distributed noise terms, which lead to dynamics driven by
standard Brownian motion (or similar L\'{e}vy processes) under
re-scaling. Arguably, this choice is often based on convenience, since the
resulting Markovian models are fairly simple to analyze. Besides, quantities
of interest in Markovian models are often relatively easy to compute, with
different numerical techniques available. On the other hand, non-trivial
correlations in the noise can lead to more general (Gaussian) processes (under
re-scaling) such as fractional Brownian motion. These models are typically
much more difficult to both analyze and compute, but may well be adequate to
explain features of the underlying real world phenomena. For instance, some
recent studies in finance report excellent fits of (simple) asset price models
driven by fractional Brownian motions to market data, substantially improving
the performance of models (even more complex ones) driven by standard Brownian
motion. We refer to \cite{GJR14a,BFG15} for more information. The importance
of fractional Brownian motion for models of turbulence has been understood
since the 1940s, see, for instance, \cite{SS94}. Rough path analysis provides
a framework for analysis and numerics for these kind of models.

In a sense, the rough path interpretation of a differential equation is
closely related to strong, pathwise error estimates of Euler-
resp.~Milstein-approximation to stochastic differential equations. For
instance, Davie's definition \cite{D07} of a (rough)pathwise SDE solution is 
\begin{equation}
Y_{t}-Y_{s}\equiv Y_{s,t}=V_{i}\left( Y_{s}\right)
B_{s,t}^{i}+V_{i}^{k}\left( Y_{s}\right) \partial _{k}V_{j}\left(
Y_{s}\right) \mathbb{B}_{s,t}^{i,j}+o\left( \left\vert t-s\right\vert
\right) \text{ as } t - s\downarrow 0,  \label{DavieRHS}
\end{equation}
where we employ Einstein's convention. In fact, this becomes an entirely
\textit{deterministic} definition, only assuming
\begin{equation*}
\exists \alpha \in \left( 1/3,1/2\right) :\left\vert B_{s,t}\right\vert \leq
C\left\vert t-s\right\vert ^{\alpha },\left\vert \mathbb{B}_{s,t}\right\vert
\leq C\left\vert t-s\right\vert ^{2\alpha },\text{ \ \ }
\end{equation*}%
something which is known to hold true almost surely (i.e. for $C=C\left(
\omega \right) <\infty $ a.s.), and something which is not at all restricted
to\ Brownian motion. As the reader may suspect this approach leads to
almost-sure convergence (with rates) of schemes which are based on the
iteration of the approximation seen in the right-hand-side of (\ref{DavieRHS}%
). The practical trouble is that L\'{e}vy's area, the anti-symmetric part of $%
\mathbb{B}$, is notoriously difficult to simulate; leave alone the
simulation of L\'{e}vy's area for other Gaussian processes. It has been
understood for a while, at least in the Brownian setting, that the truncated
(or: simplified) Milstein scheme, in which L\'{e}vy's area is omitted, i.e.
replace $\mathbb{B}_{s,t}$ by $Sym\left( \mathbb{B}_{s,t}\right) $ in (\ref%
{DavieRHS}), still offers benefits: For instance, Talay~\cite{Tal86}
replaces L\'{e}vy area by suitable Bernoulli r.v. such as to obtain weak
order $1$ (see also Kloeden--Platen \cite{KP92} and the references therein).%
\footnote{%
A well-known counter-example by Clark and Cameron \cite{CC80} shows that it
is impossible to get strong order $1$ if only using Brownian increments.} In
the multilevel context, \cite{GS14} use this truncated Milstein scheme
together with a sophisticated antithetic (variance reduction) method.
Finally, in the rough path context this scheme was used in \cite{DNT12}: the
convergence of the scheme can be traced down to an underlying Wong-Zakai
type approximation for the driving random rough path -- a (probabilistic!)
result which is known to hold in great generality for stochastic processes,
starting with \cite{CQ02} in the context of fractional Brownian motion, see 
\cite[Ch.~15]{FV10} and the references therein.

A rather difficult problem is to go from almost-sure convergence (with
rates) to $L^{1}$ (or ever: $L^{r}$ any $r<\infty $) convergence. Indeed, as
pointed out in \cite[Remark 1.2]{DNT12}: \textit{"Note that the almost sure
estimate [for the simplified Milstein scheme] cannot be turned into an }$%
L^{1}$\textit{-estimate [...]. This is a consequence of the use of the rough
path method, which exhibits non-integrable (random) constants."} The
resolution of this problem forms the first contribution of this paper. It is
based on some recent progress \cite{CLL13}, see also \cite{FR12b}, initially developed to
prove smoothness of laws for (non-Markovian) SDEs driven by Gaussian signals
under a H\"{o}rmander condition, \cite{CF10, HP13}.

Having established $L^{r}$-convergence (any $r<\infty $, with rates) for
implementable ``simplified'' Milstein schemes we move to the second
contribution of the paper: a multilevel algorithm, in the sense of Giles \cite{G08b}, for
stochastic differential equations driven by large classes of Gaussian
signals. 

The savings here are rather dramatic. In absence of Markovian structure, the
strong rate must proxy for the weak rate, which leaves one with complexity
$\mathcal{O}(\epsilon^{-\theta-2})$, any $\theta > 2 \rho / ( 2- \rho)$ where
the parameter $\rho$ quantifies the roughness of the noise. With multilevel,
we reduce this to $\mathcal{O}(\epsilon^{-\theta})$. For instance, in the case
of fBm with $H=0.4$, one has $\rho = 1.25$ and thus $\theta \sim 3.33$, which
is only marginally worse than single level Monte Carlo in the
  Brownian motion context (where one has $\theta = 3$). On the other hand, the
  computational cost of single-level Monte Carlo would be proportional to
  $\varepsilon^{-\theta-2} \sim \varepsilon^{-5.33}$ which corresponds to a
numerically non-feasible situation.

A \textit{strong}, $L^{2}$ error estimate (``rate $%
\beta /2$'') is the key assumption in Giles' complexity theorem, and this is
precisely what we have established in the first part. Some other extension of
the Giles theorem are necessary; indeed it is crucial to allow for a weak rate
of convergence $\alpha <1/2$ (ruled out explicitly in \cite{G08b}) whenever we
deal with driving signals with sample path regularity ``worse'' then Brownian
motion. We are able to do all this, moreover we carefully keep track of the
relevant constants in front of the asymptotic terms, a necessity in such an
irregular regime. % (This also allows us to recover the logarithmic terms in the
% Brownian motion context; see Theorem~\ref{thr:optimal-L-beta-1}.)

Let us now discuss the algorithm in more detail. We consider the following
scheme for approximating $Y$, see~\cite{DNT12,FV10}. Given an equi-distant
dissection $D=\left( t_{k}\right) $ of $\left[ 0,T\right] $ with mesh $h$, so
that $t_{k+1}-t_{k}\equiv h$ for all $k$, write $X_{t_{k},t_{k+1}}$ for the
corresponding increments. We then define $\overline{Y}_{0}\equiv Y_{0}$ and
the ``simplified'' step-$3$ Euler scheme
\begin{equation}
\overline{Y}_{t_{k+1}}=\overline{Y}_{t_{k}}+\sum_{l=1}^{3}\frac{1}{l!}%
V_{i_{1}}\cdots V_{i_{l}}I\left( \overline{Y}_{t_{k}}\right)
X_{t_{k},t_{k+1}}^{i_{1}}\cdots X_{t_{k},t_{k+1}}^{i_{l}},
\label{eq:simplified-stepN-euler}
\end{equation}%
where $I(y)=y$ is the identity function and the vector fields $V_{1},\ldots
,V_{d}$, unless otherwise stated assumed bounded with bounded derivatives of all orders, are viewed as linear first order operators. % , acting on functions by $%
% V_{i}g(y)=\nabla g(y)\cdot V_{i}(y)$
%{\color{blue} and bounded and $C^\infty$ with bounded derivatives}. 
Whenever
convenient we extend $\overline{Y}$ to $\left[ 0,T\right] $ by linear
interpolation. Moreover, the Einstein summation convention is in force. For a
more detailed description of the algorithm we refer to Section
\ref{sec:prob-conv-results}. We are now able to state (a simple version of) our main results;
cf. Corollary \ref{cor:strong_rates_smooth_vf}:

\begin{theorem}[Strong rates]
\label{Thm1Intro}Let $X=\left( X^{1},\dots ,X^{d}\right) $ be a continuous,
zero-mean Gaussian process with independent components. Assume furthermore
that each component has stationary increments and that%
\begin{equation*}
\sigma ^{2}\left( t-s\right) :=E\left\vert X_{t}^{i}-X_{s}^{i}\right\vert
^{2}
\end{equation*}%
where $\sigma ^{2}$ is concave and $\sigma ^{2}\left( \tau \right) =\mathcal{O}\left(
\tau ^{1/\rho }\right) $ as $\tau \rightarrow 0$ for some $\rho \in \lbrack
1,2).$\newline
Let $Y$ be the solution to the rough differential equation (\ref{RDEIntro})
driven by (the rough path lift) of $X$ and $\overline{Y}=\overline{Y}%
^{h}$ be the approximate solution based on (\ref%
{eq:simplified-stepN-euler}). Then we have strong convergence of (almost)
rate $1/\rho -1/2$. More precisely, for any $1\leq r<\infty $ and $\delta >0$%
, there exists a constant $C$ such that  
\begin{equation*}
\left\vert E\left( \sup_{t\in \left[ 0,T\right] }\left\vert Y_{t}-\overline{Y%
}_{t}^{h}\right\vert ^{r}\right) \right\vert ^{\frac{1}{r}}\leq
C h^{1/\rho -1/2-\delta }.
\end{equation*}
\end{theorem}

The reader should notice that the assumption on $X$ is met by
multidimensional Brownian motion (with $\rho =1$) in which case $Y$ is
nothing but a Stratonovich solution of the SDE (\ref{SDEIntro}), which of
course may be rewritten as It\^{o} equation. More interestingly, $X$ may be
a fractional Brownian motion (with $\rho = \frac{1}{2H}>1$) in the (interesting)
``rougher than Brownian'' regime $H\in (1/4,1/2)$. However,
  stationarity of increments of $X$ plays very little role aside from allowing
an easy-to-state formulation above. The precise technical requirement is given
in Condition~\ref{mixed_hoel} below and is satisfied in many examples, see \cite{FGGR16}.

Using Giles' multi-level
Monte Carlo methodology, we can greatly improve the complexity bounds for the
discretization algorithm~\eqref{eq:simplified-stepN-euler}, see
Theorem~\ref{thr:fBM-ml-comp}. Note that the overall
  computational cost depends on the computational cost of simulating the
  increments of the process $X$ as needed in
  equation~\eqref{eq:simplified-stepN-euler}. For each trajectory of $Y$, we
  need to sample a vector of increments of $X$ of length $N \equiv T/h$,
  i.e., we need to sample from an $N$-dimensional normal distribution with
  known non-diagonal covariance matrix. In general, this can be achieved at
  cost proportional to $N^2$ by multiplication of a standard normal vector
  with the lower triangular factor obtained by the Cholesky factorization of
  the covariance matrix. In many special cases, for instance for fBm, the cost
  can be reduced to $N\log N$, see~\cite{D04}, and we concentrate on this
  scenario below. The precise statement is given in
  Theorem~\ref{thr:fBM-ml-comp}.

\begin{theorem}[Multilevel complexity estimate]
  \label{Thm2Intro}Let $X$ and $Y$ be as in the previous theorem and $%
  f:C([0,T],\R^{m})\rightarrow \R^{n}$ a Lipschitz continuous
  functional. Assume that the computational cost of generating a
    vector of (non-overlapping) increments of $X$ of length $N$ is
    proportional to $N \log N$. Then the Monte Carlo evaluation of a
  path-dependent functional of the form%
\begin{equation*}
E(f(Y_{t}:0\leq t\leq T))
\end{equation*}
to within a MSE of $\varepsilon^{2}$, can be achieved with computational
work
\begin{equation*}
  \mathcal{O}\left( \varepsilon ^{-\theta }\right), \quad \forall \theta
  >\frac{2\rho}{2-\rho}.
\end{equation*}
\end{theorem}

As a sanity check, let us compare this results with the corresponding,
well-known results for classical stochastic differential equations (here: in
Stratonovich sense) driven by $d$-dimensional Brownian motion $B$.  The
assumptions on $X$ are clearly met with $\rho =1$. As a consequence, we obtain
strong convergence of (almost) rate $1/2$ in agreement with the well-known
strong rate $1/2$ in the classical setting. Concerning our multilevel
complexity estimate, we obtain (almost) order $\varepsilon ^{-2}$ which is
arbitrarily "close" to known result $\mathcal{O}\left( \varepsilon ^{-2}\left(
    \log \varepsilon \right) ^{2}\right) $ \cite{G08a, G08b}, recently
sharpened to $\mathcal{O}\left( \varepsilon ^{-2}\right) $ \cite{GS14} with
the aid of a suitable antithetic multilevel correction estimator.

Let us summarize the (computational) benefits of the multilevel approach in the
present (``rougher than Brownian'') setting. A direct Monte
Carlo implementation of the scheme~\eqref{eq:simplified-stepN-euler} would
require a complexity of $\mathcal{O}(\varepsilon ^{-(2+1/\alpha )})$ in order to
attain an MSE of no more than $\varepsilon ^{2}$. Here, $\alpha $ is the
\emph{weak} rate of convergence of the scheme. On the other hand, we show in
Theorem~\ref{thr:ml-complexity} that the complexity is only
$\mathcal{O}(\varepsilon^{-(1+2\alpha-\beta)/\alpha})$ for the multi-level Monte
Carlo estimator, where $\beta$ is two times the strong rate of
convergence. Thus, when the weak rate of convergence is equal to the strong
rate of convergence\footnote{By lack of the Markov property, the standard
  techniques of deriving weak error estimates fail in the setting of an RDE
  driven by a general Gaussian process such as a fBm. Thus, computing the weak
  rate of convergence for the simplified Euler scheme would be a non-trivial
  task. On the other hand, we present a numerical example in
  Section~\ref{sec:numer-exper}, where the weak order is equal to the strong
  order even in a standard Brownian motion setting.}, then the complexity of
the multi-level estimator is reduced by a factor $\varepsilon^2$ as compared by
the complexity of the standard Monte Carlo estimator. When the weak rate is
two times the strong rate, the speed up is still by a factor $\varepsilon$, see
Table~\ref{tab:summary-general} and Table~\ref{tab:complexity-summ-fbm}.

\section{Rough path estimates revisited}\label{sec:rp_est_rev}

In this section, we revisit some classical estimates used in rough paths
theory. Definitions of the basic objects and all relevant notation may be
found in the appendix. A more detailed account to the theory of rough paths may be found in the monographs \cite{LQ02}, \cite{LCL07}, \cite{FV10} or  \cite{FH14}. 

Versions of the results we are interested in (cf. the forthcoming theorems \ref{cor_loc_lip_integrability_unif_top} and \ref{prop_euler_rates}) are already stated in the above mentioned references, and the given estimates are (essentially) sharp when the oscillations of the driving rough path (i.e. its $p$-variation) become small. However, it turns out that they are less useful when its oscillations get large. In this case, we will show that the estimates can be improved by replacing the occurring $p$-variation norm of the rough path $\mathbf{x}$ in the inequality by another quantity $N(\mathbf{x})$ which was first introduced by Cass-Litterer-Lyons in \cite{CLL13} (and which we recall in Definition \ref{def:greedy_partitions} below). This becomes crucial when substituting the deterministic rough path $\mathbf{x}$ by the lift of a Gaussian process $\mathbf{X}$: A key result in \cite{CLL13} states that the quantity $N(\mathbf{X})$ enjoys significantly better integrability properties than the $p$-variation of $\mathbf{X}$. 

The aim of this section is to show that one can indeed improve the bound for the Lipschitz constant of the It\^o-Lyons map (Theorem \ref{cor_loc_lip_integrability_unif_top}) and the estimate for the distance of the solution of a rough differential equation (RDE) and its Euler (or Milstein) approximation (Theorem \ref{prop_euler_rates}). This will allow us to deduce the desired probabilistic estimates in the forthcoming section \ref{sec:prob_conv_for_rdes}.

\subsection{Improved bounds for the Lipschitz constant of the It\^o-Lyons map}

% It is well-known that the Lipschitz constant of the It\^o-Lyons solution map for an RDE driven by a rough path $\mathbf{x}$ is of the order $\mathcal{O}(\exp(C \| \mathbf{x} \|_{p-\text{var}}^p))$. Considering Gaussian driving signals, this (random) constant fails to have finite $q$-th moments for any $q$. In this section, we improve the deterministic estimates for the Lipschitz constant slightly which will allow us to derive the desired estimates.

% The aim of this section is to improve known estimates for the Lipschitz constant of the It\^o-Lyons solution map for an RDE driven by a rough path.

Recall the following definition, taken from \cite{CLL13}:

\begin{definition}\label{def:greedy_partitions}
Let $\omega $ be a control function, that is, a continuous function $\omega \colon \{(s,t)\,:\, 0 \leq s \leq t \leq T\} \to [0,\infty)$ for which $\omega(s,t) + \omega(t,u) \leq \omega(s,u)$ holds for every $s \leq t \leq u$. For $\alpha >0$ and $\left[ s,t\right]
\subset \left[ 0,T\right] $, we set 
\begin{eqnarray*}
\tau _{0}\left( \alpha \right)  &=&s \\
\tau _{i+1}\left( \alpha \right)  &=&\inf \left\{ u:\omega \left( \tau
_{i}(\alpha),u\right) \geq \alpha ,\tau _{i}\left( \alpha \right) <u\leq t\right\}
\wedge t
\end{eqnarray*}%
and define 
\begin{equation*}
N_{\alpha ,\left[ s,t\right] }\left( \omega \right) =\sup \left\{ n\in 
\mathbb{N\cup }\left\{ 0\right\} :\tau _{n}\left( \alpha \right) <t\right\} .
\end{equation*}%
When $\omega $ arises from the (homogenous) $p$-variation norm $\Vert \cdot
\Vert _{p-\text{var}}$ of a ($p$-rough) path, $\mathbf{x}$, i.e. $\omega _{%
\mathbf{x}}=\left\Vert \mathbf{x}\right\Vert _{p\text{-var;}\left[ \cdot
,\cdot \right] }^{p}$ with $p \geq 1$, we shall also write $N_{\alpha ,\left[ s,t\right]
}\left( \mathbf{x}\right) :=N_{\alpha ,\left[ s,t\right] }\left( \omega _{%
\mathbf{x}}\right) $.
\end{definition}

It is easy to see that $\alpha N_{\alpha ,[ 0,T] }\left( \mathbf{x}\right) \leq \| \mathbf{x} \|_{p-\text{var};[0,T]}^p$ (\cite[Lemma 4.9]{CLL13}), and this is sharp (as one can see choosing $\alpha \nearrow \| \mathbf{x} \|_{p-\text{var};[0,T]}^p$). However, for fixed $\alpha$, the tail estimates for $N_{\alpha ,[ 0,T] }\left( \mathbf{X}\right)$ are significantly better than for $\| \mathbf{X} \|_{p-\text{var};[0,T]}^p$ when we consider Gaussian lifts $\mathbf{X}$, cf. \cite{CLL13} and \cite{FR12b}. 

Next, we give the main result of this section. The following theorem is a variant of \cite[Theorem 10.38]{FV10}. The main difference is that in \cite[Theorem 10.38]{FV10}, the Lipschitz constant is (essentially) given by $C \exp \left\{ C\left( \| \mathbf{x}^{1}\|_{p-\text{var};[0,T]}^p + \| \mathbf{x}^{2}\|_{p-\text{var};[0,T]}^p \right) \right\}$, whereas in the following theorem, $\|\mathbf{x}^{i}\|_{p-\text{var};[0,T]}^p$ is replaced by $N_{\alpha ,[0,T]}(\mathbf{x}^{i})$, $i = 1,2$.

\begin{theorem}
\label{cor_loc_lip_integrability_unif_top} Consider the RDEs%
\begin{equation*}
dy_{t}^{i}=V^{i}(y_{t}^{i})\,d\mathbf{x}_{t}^{i};\quad y_{0}^{i}\in \mathbb{R%
}^{e}
\end{equation*}%
for $i=1,2$ on $[0,T]$ where $V^{1}$ and $V^{2}$ are two families of vector
fields, $\gamma >p$ and $\nu $ is a bound on $|V^{1}|_{Lip^{\gamma }}$ and $%
|V^{2}|_{Lip^{\gamma }}$. Then for every $\alpha >0$ there is a constant $%
C=C(\gamma ,p,\nu ,\alpha )$ such that%
\begin{eqnarray*}
\left\vert y^{1}-y^{2}\right\vert _{\infty \text{;}\left[ 0,T\right] } &\leq
&\ C\left[ |y_{0}^{1}-y_{0}^{2}|+\left\vert V^{1}-V^{2}\right\vert _{\text{%
Lip}^{\gamma -1}}+\rho _{p-\text{var};[0,T]}(\mathbf{x}^{1},\mathbf{x}^{2})%
\right] \\
&&\times \exp \left\{ C\left( N_{\alpha ,[0,T]}(\mathbf{x}^{1})+N_{\alpha
,[0,T]}(\mathbf{x}^{2})\right) \right\}
\end{eqnarray*}%
holds.
\end{theorem}

The proof of Theorem \ref{cor_loc_lip_integrability_unif_top} will be given at the end of this section. We first prove some preparatory lemmata. Recall that if $\omega^1$ and $\omega^2$ are controls, also $\omega^1 +
\omega^2$ is a control.

\begin{lemma}
\label{lemma_Nsumofcontrols} Let $\omega^1$ and $\omega^2$ be two controls.
Then 
\begin{align*}
N_{\alpha,[s,t]}(\omega^1 + \omega^2) \leq 2N_{\alpha,[s,t]}(\omega^1) +
2N_{\alpha,[s,t]}(\omega^2) + 2
\end{align*}
for every $s<t$ and $\alpha>0$.
\end{lemma}

\begin{proof}
If $\omega $ is any control, set
\begin{equation*}
\omega _{\alpha }\left( s,t\right) := \sup \left\{ \sum_{i = 0}^{M-1} \omega (t_{i},t_{i+1})\, :\, s = t_0 < t_1 < \ldots < t_M = t,\ \omega(t_i,t_{i+1}) \leq \alpha,\ n \in \N \right\}.
%:=\sup_{\substack{ \left( t_{i}\right)
%=D\subset \left[ s,t\right]  \\ \omega (t_{i},t_{i+1})\leq \alpha }}%
%\sum_{t_{i}}\omega (t_{i},t_{i+1}).
\end{equation*}%
If $\bar{\omega}:=\omega ^{1}+\omega ^{2}$, $\bar{\omega}(t_{i},t_{i+1})\leq
\alpha $ implies $\omega ^{i}(t_{i},t_{i+1})\leq \alpha $ for $i=1,2$ and
therefore $\bar{\omega}_{\alpha }\left( s,t\right) \leq \omega _{\alpha
}^{1}\left( s,t\right) +\omega _{\alpha }^{2}\left( s,t\right) $. From
Proposition 4.6 in \cite{CLL13} we know that $\omega _{\alpha }^{i}\left(
s,t\right) \leq \alpha \left( 2N_{\alpha ,\left[ s,t\right] }\left( \omega
^{i}\right) +1\right) $ for $i=1,2$. 
% (Strictly speaking, Proposition 4.6 is formulated for a particular control $\omega $, namely the control induced by
% the $p$-variation of a rough path. However, the proof only uses general
% properties of control functions and the conclusion remains valid.) 
We conclude
\begin{align*}
\alpha N_{\alpha ,\left[ s,t\right] }\left( \bar{\omega}\right) &
=\sum_{i=0}^{N_{\alpha ,\left[ s,t\right] }\left( \bar{\omega}\right) -1}%
\bar{\omega}(\tau _{i}\left( \alpha \right) ,\tau _{i+1}\left( \alpha
\right) ) \leq \bar{\omega}_{\alpha }(s,t) \leq \omega _{\alpha }^{1}\left( s,t\right) +\omega _{\alpha }^{2}\left(
s,t\right) \\
&\leq \alpha \left( 2N_{\alpha ,\left[ s,t\right] }\left( \omega
^{1}\right) +2N_{\alpha ,\left[ s,t\right] }\left( \omega ^{2}\right)
+2\right) .
\end{align*}
\end{proof}

\begin{lemma}
\label{lemma_Nsumofcontrolssecondbdd} Let $\omega^1$ and $\omega^2$ be two
controls and assume that $\omega^2(s,t)\leq K$. Then 
\begin{align*}
N_{\alpha,[s,t]}(\omega^1 + \omega^2) \leq N_{\alpha - K,[s,t]}(\omega^1)
\end{align*}
for every $\alpha>K$.
\end{lemma}

\begin{proof}
Set $\bar{\omega}:=\omega ^{1}+\omega ^{2}$ and
\begin{eqnarray*}
\bar{\tau}_{0}\left( \alpha \right) &=&s \\
\bar{\tau}_{i+1}\left( \alpha \right) &=&\inf \left\{ u:\bar{\omega}\left(
\bar{\tau}_{i}(\alpha),u\right) \geq \alpha ,\bar{\tau}_{i}\left( \alpha \right)
<u\leq t\right\} \wedge t.
\end{eqnarray*}%
Similarly, we define $(\tau _{i})_{i\in \mathbb{N}}=(\tau _{i}(\alpha
-K))_{i\in \mathbb{N}}$ for $\omega ^{1}$. It suffices to show that $\bar{%
\tau}_{i}\geq \tau _{i}$ for $i=0,\ldots ,N_{\alpha ,[s,t]}(\bar{\omega})$.
We do this by induction. For $i=0$, this is clear. If $\bar{\tau}_{i}\geq
\tau _{i}$ for some $i\leq N_{\alpha ,[s,t]}(\bar{\omega})-1$,
superadditivity of control functions gives
\begin{equation*}
\alpha =\bar{\omega}(\bar{\tau}_{i},\bar{\tau}_{i+1})\leq \omega ^{1}(\tau
_{i},\bar{\tau}_{i+1})+K
\end{equation*}%
which implies $\tau _{i+1}\leq \bar{\tau}_{i+1}$.
\end{proof}

For the next Lemma, recall the definition of the homogenous $p$-$\omega$ distance and -norm given in the appendix and in  \cite[Definition 8.2]{FV10}.
\begin{lemma}
\label{lemma_boundinftydist} Let $s<t\in \lbrack 0,T]$ and assume that $%
\Vert \mathbf{x}^{i}\Vert _{p-\omega ;[s,t]}\leq 1$ for $i=1,2$. Then there
is a constant $C=C(\gamma ,p)$ such that 
\begin{align*}
\nu |y^{1}-y^{2}|_{\infty ;[s,t]}\leq & \left[ \nu
|y_{s}^{1}-y_{s}^{2}|+\left\vert V^{1}-V^{2}\right\vert _{\text{Lip}^{\gamma
-1}}+\nu \rho _{p-\omega ;[s,t]}(\mathbf{x}^{1},\mathbf{x}^{2})\right] \\
& \times (N_{\alpha ,[s,t]}(\omega )+1)\exp \left\{ C\nu ^{p}\alpha
(N_{\alpha ,[s,t]}(\omega )+1)\right\}
\end{align*}%
for every $\alpha >0$.
\end{lemma}

\begin{proof}
Set $\bar{y}=y^{1}-y^{2}$ and
\begin{equation*}
\kappa =\frac{\left\vert V^{1}-V^{2}\right\vert _{\text{Lip}^{\gamma -1}}}{%
\nu }+\rho _{p-\omega ;[s,t]}(\mathbf{x}^{1},\mathbf{x}^{2}).
\end{equation*}%
From \cite[Theorem 10.26]{FV10} we can deduce that there is a constant $%
C=C(\gamma ,p)$ such that
\begin{equation*}
|\bar{y}_{u,v}|\leq C\nu \omega (u,v)^{1/p}\left[ |\bar{y}_{u}|+\kappa %
\right] \exp \left\{ C\nu ^{p}\omega (u,v)\right\}
\end{equation*}%
for every $u<v\in \lbrack s,t]$. From $|\bar{y}_{u,v}|\geq |\bar{y}_{s,v}|-|%
\bar{y}_{s,u}|$ we obtain
\begin{align*}
|\bar{y}_{s,v}|& \leq C\nu \omega (u,v)^{1/p}\left[ |\bar{y}_{u}|+\kappa %
\right] \exp \left\{ C\nu ^{p}\omega (u,v)\right\} +|\bar{y}_{s,u}| \\
& \leq \left[ |\bar{y}_{s}|+|\bar{y}_{s,u}|+\kappa \right] \exp \left\{ C\nu
^{p}\omega (u,v)\right\}
\end{align*}%
for $s\leq u<v\leq t$. Now let $s=\tau _{0}<\tau _{1}<\ldots <\tau _{M}<\tau
_{M+1}=v\leq t$ for $M\geq 0$. By induction, one sees that
\begin{align*}
|\bar{y}_{s,v}|& \leq (M+1)(|\bar{y}_{s}|+\kappa )\exp \left\{ C\nu
^{p}\sum_{i=0}^{M}\omega (\tau _{i},\tau _{i+1})\right\} \\
& \leq C^{M+1}\left[ |\bar{y}_{s}|+\kappa \right] \exp \left\{ C\nu
^{p}\sum_{i=0}^{M}\omega (\tau _{i},\tau _{i+1})\right\} .
\end{align*}%
It follows that for every $v\in \lbrack s,t]$,
\begin{equation*}
|\bar{y}_{s,v}|\leq \left[ |\bar{y}_{s}|+\kappa \right] (N_{\alpha
,[s,t]}(\omega )+1)\exp \left\{ C\nu ^{p}\alpha (N_{\alpha ,[s,t]}(\omega
)+1)\right\} ,
\end{equation*}%
therefore
\begin{equation*}
|\bar{y}_{v}|\leq \left[ |\bar{y}_{s}|+\kappa \right] (N_{\alpha
,[s,t]}(\omega )+1)\exp \left\{ C\nu ^{p}\alpha (N_{\alpha ,[s,t]}(\omega
)+1)\right\} +|\bar{y}_{s}|
\end{equation*}%
and finally
\begin{equation*}
|\bar{y}|_{\infty ;[s,t]}\leq \left[ |\bar{y}_{s}|+\kappa \right] (N_{\alpha
,[s,t]}(\omega )+1)\exp \left\{ C\nu ^{p}\alpha (N_{\alpha ,[s,t]}(\omega
)+1)\right\} .
\end{equation*}
\end{proof}

\begin{proof}[Proof of Theorem \ref{cor_loc_lip_integrability_unif_top}]
Let $\omega $ be a control such that $\Vert \mathbf{x}^{i}\Vert _{p-\omega
;[0,T]}\leq 1$ for $i=1,2$ (the precise choice of $\omega $ will be made
later). From Lemma \ref{lemma_boundinftydist} we know that there is a constant $%
C=C(\gamma ,p,\nu ,\alpha )$ such that
\begin{align*}
\left\vert y^{1}-y^{2}\right\vert _{\infty \text{;}\left[ 0,T\right] }& \leq %
\left[ |y_{0}^{1}-y_{0}^{2}|+\left\vert V^{1}-V^{2}\right\vert _{\text{Lip}%
^{\gamma -1}}+\rho _{p-\omega ;[s,t]}(\mathbf{x}^{1},\mathbf{x}^{2})\right]
\\
& \times \exp \left\{ C(N_{\alpha ,[s,t]}(\omega )+1)\right\} .
\end{align*}%
Now we set $\omega =\omega _{\mathbf{x}^{1},\mathbf{x}^{2}}$ where
\begin{equation*}
\omega _{\mathbf{x}^{1},\mathbf{x}^{2}}(s,t)=\Vert \mathbf{x}^{1}\Vert _{p-%
\text{var};[s,t]}^{p}+\Vert \mathbf{x}^{2}\Vert _{p-\text{var}%
;[s,t]}^{p}+\sum_{k=1}^{\lfloor p\rfloor }\frac{\left( \rho _{p-\text{var}%
;[s,t]}^{(k)}(\mathbf{x}^{1},\mathbf{x}^{2})\right) ^{p/k}}{\left( \rho _{p-%
\text{var};[0,T]}^{(k)}(\mathbf{x}^{1},\mathbf{x}^{2})\right) ^{p/k}}
\end{equation*}%
(the definition of $\rho_{p-\text{var}}^{(k)}(\cdot,\cdot)$ may be found in the appendix). It is easy to check that
\begin{align*}
\Vert \mathbf{x}^{1}\Vert _{p-\omega _{\mathbf{x}^{1},\mathbf{x}
^{2}};[0,T]}& \leq 1,\quad \Vert \mathbf{x}^{2}\Vert _{p-\omega _{\mathbf{x}^{1},\mathbf{x}%
^{2}};[0,T]} \leq 1\quad \text{ and} \\
\rho _{p-\omega _{\mathbf{x}^{1},\mathbf{x}^{2}};[0,T]}(\mathbf{x}^{1},%
\mathbf{x}^{2})& \leq \rho _{p-\text{var};[0,T]}(\mathbf{x}^{1},\mathbf{x}%
^{2}).
\end{align*}%
Finally, if $\alpha >\lfloor p\rfloor $ we can use Lemma \ref%
{lemma_Nsumofcontrolssecondbdd} and Lemma \ref{lemma_Nsumofcontrols} to see
that
\begin{align*}
N_{\alpha ,[0,T]}(\omega _{\mathbf{x}^{1},\mathbf{x}^{2}})+1& \leq N_{\alpha
-\lfloor p\rfloor ,[0,T]}(\omega _{\mathbf{x}^{1}}+\omega _{\mathbf{x}%
^{2}})+1 \\
& \leq 3\left( N_{\alpha -\lfloor p\rfloor ,[0,T]}(\mathbf{x}^{1})+N_{\alpha
-\lfloor p\rfloor ,[0,T]}(\mathbf{x}^{2})+1\right) .
\end{align*}%
Substituting $\alpha \mapsto \alpha +\lfloor p\rfloor $ gives the claimed
estimate.
\end{proof}

\subsection{Improved bounds for the higher order Euler approximations}

We are now interested in proving a similar estimate for the distance between Euler-/Milstein approximations for rough paths and the actual solution (for
the purpose of unified terminology, in the sequel we will only speak of
Euler-schemes). Recall the notation from \cite{FV10}: If $V=\left(
V_{1},\ldots ,V_{d}\right) $ is a collection of sufficiently smooth vector
fields on $\mathbb{R}^{e}$, $g\in T^{N}\left( \mathbb{R}^{d}\right) $ and $%
y\in \mathbb{R}^{e}$, we define an increment of the step-$N$ Euler scheme by%
\begin{equation*}
\mathcal{E}_{\left( V\right) }\left( y,g\right)
:=\sum_{k=1}^{N}V_{i_{1}}\ldots V_{i_{k}}I\left( y\right) g^{k,i_{1},\ldots
,i_{k}}
\end{equation*}%
where $g^{k,i_{1},\ldots ,i_{k}}=\pi _{k}\left( g\right) ^{i_{1},\ldots
,i_{k}}\in \mathbb{R}$, $I$ is the identity on $\mathbb{R}^{e}$ and every $%
V_{j}$ is identified with the first-order differential operator $%
V_{j}^{k}\left( y\right) \frac{\partial }{\partial y^{k}}$ (throughout, we
use the Einstein summation convention). Furthermore, we set%
\begin{equation*}
\mathfrak{E}^{g}y:=y+\mathcal{E}_{\left( V\right) }\left( y,g\right) .
\end{equation*}%
Given $D=\left\{ 0=t_{0}<\ldots <t_{n}=T\right\} $ and a path $\mathbf{x}\in
C^{p-var}_0\left( \left[ 0,T\right] ;G^{\lfloor p\rfloor }\left( \mathbb{R}%
^{d}\right) \right) $ we define the (step-$N$) Euler approximation to the
RDE solution $y$ of%
\begin{equation}
dy=V(y)\,d\mathbf{x}  \label{eqn_rde}
\end{equation}%
with starting point $y_{0}\in \mathbb{R}^{e}$ at time $t_{k}\in D$ by%
\begin{equation*}
y_{t_{k}}^{\text{Euler;}D}:=\mathfrak{E}^{t_{k}\leftarrow t_{0}}y_{0}:=%
\mathfrak{E}^{S_{N}\left( \mathbf{x}\right) _{t_{k-1},t_{k}}}\circ \cdots
\circ \mathfrak{E}^{S_{N}\left( \mathbf{x}\right) _{t_{0},t_{1}}}y_{0}
\end{equation*}
 where $S_{N}\left( \mathbf{x}\right)$ denotes the Lyons lift of the rough path $\mathbf{x}$, see \cite[Section 9.1]{FV10}.
 
 The following theorem is a version of \cite[Theorem 10.30]{FV10} where, as in Theorem \ref{cor_loc_lip_integrability_unif_top}, the estimate is improved by replacing the quantity $\| \cdot \|_{p-\text{var};[0,T]}^p$ by $N_{\alpha,[0,T]}(\cdot)$.

\begin{theorem}
\label{prop_euler_rates}Let $\mathbf{x}\in C^{p-var}_0\left( \left[ 0,T\right]
;G^{\lfloor p\rfloor }\left( \mathbb{R}^{d}\right) \right) $ and set $\omega
\left( s,t\right) =\left\Vert \mathbf{x}\right\Vert _{p-\text{var};\left[ s,t%
\right] }^{p}$. Assume that $V\in Lip^{\theta }$ for some $\theta >p$ and
let $\nu \geq \left\vert V\right\vert _{Lip^{\theta }}$. Choose $N\in 
\mathbb{N}$ such that $\lfloor p\rfloor \leq N\leq \theta $. Fix a
dissection $D=\left\{ 0=t_{0}<\ldots <t_{n}=T\right\} $ of $\left[ 0,T\right]
$ and let $y_{T}^{\text{Euler;}D}$ denote the step-$N$ Euler approximation
of $y$. Then for every $\zeta \in \left[ \frac{N}{p},\frac{N+1}{p}\right) $
and $\alpha >0$ there is a constant $C=C\left( p,\theta ,\zeta ,N,\nu
,\alpha \right) $ such that%
\begin{equation*}
\left\vert y_{T}-y_{T}^{\text{Euler;}D}\right\vert \leq C\exp \left\{
C\left( N_{\alpha ,[0,T]}(\mathbf{x})+1\right) \right\} \sum_{k=1}^{n}\omega
\left( t_{k-1},t_{k}\right) ^{\zeta }.
\end{equation*}%
In particular, if $\mathbf{x}$ is a H\"{o}lder rough path and $\left\vert
t_{k+1}-t_{k}\right\vert \leq \left\vert D\right\vert $ for all $k$ we obtain%
\begin{equation}
\left\vert y_{T}-y_{T}^{\text{Euler;}D}\right\vert \leq CT\left\Vert \mathbf{%
x}\right\Vert _{1/p\text{-H\"{o}l};\left[ 0,T\right] }^{\zeta p}\exp \left\{
C\left( N_{\alpha ,[0,T]}(\mathbf{x})+1\right) \right\} \left\vert
D\right\vert ^{\zeta -1}  \label{eqn_euler_estim_hoelder_case}
\end{equation}
\end{theorem}

\begin{proof}
We basically repeat the proof of \cite[Theorem 10.30]{FV10}. Recall the
notation $\pi _{\left( V\right) }\left( s,y_{s};\mathbf{x}\right) $ for the
(unique) solution of $\left( \ref{eqn_rde}\right) $ with starting point $%
y_{s}$ at time $s$. Set%
\begin{equation*}
z^{k}=\pi _{\left( V\right) }\left( t_{k},\mathfrak{E}^{t_{k}\leftarrow
t_{0}}y_{0};\mathbf{x}\right) .
\end{equation*}%
Then $z_{t}^{0}=y_{t}$, $z_{t_{k}}^{k}=\mathfrak{E}^{t_{k}\leftarrow
t_{0}}y_{0}$ for every $k=1,\ldots ,n$ and $z_{T}^{n}=y_{T}^{\text{Euler;}D}$%
, hence%
\begin{equation*}
\left\vert y_{T}-y_{T}^{\text{Euler;}D}\right\vert \leq
\sum_{k=1}^{n}\left\vert z_{T}^{k}-z_{T}^{k-1}\right\vert .
\end{equation*}%
One can easily see that%
\begin{equation*}
z_{T}^{k-1}=\pi _{\left( V\right) }\left( t_{k-1},z_{t_{k-1}}^{k-1};\mathbf{x%
}\right) =\pi _{\left( V\right) }\left( t_{k},z_{t_{k}}^{k-1};\mathbf{x}%
\right)
\end{equation*}%
for all $k=1,\ldots ,n$. Applying Theorem \ref%
{cor_loc_lip_integrability_unif_top} (in particular the Lipschitzness in the
starting point) we obtain for any $\alpha >0$%
\begin{equation*}
\left\vert z_{T}^{k}-z_{T}^{k-1}\right\vert \leq c_{1}\left\vert
z_{t_{k}}^{k}-z_{t_{k}}^{k-1}\right\vert \exp \left\{ c_{1}\left( N_{\alpha
,[0,T]}(\mathbf{x})+1\right) \right\} .
\end{equation*}%
Moreover (cf. \cite[Theorem 10.30]{FV10}),
\begin{equation*}
\left\vert z_{t_{k}}^{k}-z_{t_{k}}^{k-1}\right\vert \leq \left\vert \pi
_{\left( V\right) }\left( t_{k-1},\cdot ,\mathbf{x}\right) _{t_{k-1},t_{k}}-%
\mathcal{E}_{\left( V\right) }\left( \cdot ,S_{N}\left( \mathbf{x}\right)
_{t_{k-1},t_{k}}\right) \right\vert _{\infty }.
\end{equation*}%
Let $\delta \in \lbrack 0,1)$ such that $\zeta =\frac{N+\delta }{p}$. Since $%
\left( N+\delta \right) -1<N\leq \gamma $ we have $V\in \text{Lip}^{\left( N+\delta
\right) -1}$. Thus we can apply \bigskip \cite[Corollary 10.15]{FV10} to see
that%
\begin{eqnarray*}
\left\vert \pi _{\left( V\right) }\left( t_{k-1},\cdot ,\mathbf{x}\right)
_{t_{k-1},t_{k}}-\mathcal{E}_{\left( V\right) }\left( \cdot ,S_{N}\left(
\mathbf{x}\right) _{t_{k-1},t_{k}}\right) \right\vert _{\infty } &\leq
&c_{2}\left( \left\vert V\right\vert _{\text{Lip}^{\left( N+\delta \right)
-1}}\left\Vert \mathbf{x}\right\Vert _{p-\text{var};\left[ t_{k-1},t_{k}%
\right] }\right) ^{N+\delta } \\
&\leq &c_{2}\left\vert V\right\vert _{\text{Lip}^{\gamma }}^{p\zeta }\omega \left(
t_{k-1},t_{k}\right) ^{\zeta }
\end{eqnarray*}%
which gives the claim.
\end{proof}

\section{Probabilistic convergence results for RDEs}\label{sec:prob_conv_for_rdes}

Let $X:[0,T]\to\R$ be a real valued, centered, continuous Gaussian process
with covariance 
\begin{equation*}
R_{X}(s,t)=E X_{s}X_{t}. 
\end{equation*}
% *** REMOVE? We will denote the associated Cameron-Martin space by $\mathcal{H}$. It is
% well-known that $\mathcal{H}\subset C([0,T],\R)$ and each $h\in\mathcal{H}$
% is of the form $h_{t}=E ZX_{t}$ with $Z$ being an element of the $L^{2}$%
% -closure of $\text{span}\{X_{t}|t\in[0,T]\}$, a Gaussian random variable. If 
% $h_{t}=E ZX_{t}$, $h^{\prime }_{t}=E Z^{\prime }X_{t}$, $<h,h^{\prime }>_{%
% \mathcal{H}}=E ZZ^{\prime }$. ***
% 
% For any function $h:[0,T]\to\R$ we define $h_{s,t}:=h_{t}-h_{s}$ for all $s,t%\in[0,T].$ 
We recall the definition of mixed right $(\gamma,\rho)$-variation:
For $\gamma,\rho\geq1$ let 
\begin{align}\label{eqn:def_mixed_rho_var}
V_{\gamma,\rho}(R_{X};[s,t]\times [u,v]):=\sup_{\substack{(t_{i})\in\mathcal{D}([s,t])\\
(t_{j}^{\prime})\in\mathcal{D}\left(\left[u,v\right]\right)
}
}\left(\sum_{t'_{j}}\left(\sum_{t_{i}}\left|R_{X}\left(\begin{array}{c}
t_{i},t_{i+1}\\
t_{j}^{\prime},t_{j+1}^{\prime}
\end{array}\right)\right|^{\gamma}\right)^{\frac{\rho}{\gamma}}\right)^{\frac{1}{\rho}},
\end{align}
where $\mcD([s,t])$ denotes the set of all dissections of $[s,t]$ and 
\begin{equation*}
R_{X}\left(%
\begin{array}{c}
t_{i},t_{i+1} \\ 
t_{j}^{\prime},t_{j+1}^{\prime}%
\end{array}%
\right)= E X_{t_{i},t_{i+1}}X_{t_{j}^{\prime},t_{j+1}^{\prime}}. 
\end{equation*}

We note that $V_{\rho} \equiv V_{\rho,\rho}$ regularity plays a key role in Gaussian rough path
theory \cite{FV10,FV10-2,FH14} and in particular yields a stochastic
integration theory for large classes of multidimensional Gaussian processes.
The importance of finite mixed $(1,\rho)$-variation was understood in \cite{FGGR16}, 
where it is shown to allow for concentration of measure results (via Cass-Litterer-Lyons \cite{CLL13})
which are pivotal for our result. In effect, one then has a substitute for good moment bounds
in It\^o theory which are no more available in our general Gaussian setting. (Recall that
fBm, other than Brownian motion, is not a semimartingale.) 

The following condition will be in place throughout the paper. 

\begin{condition}\label{mixed_p-var} 
  Let $X = (X^1,\ldots,X^d)\colon [0,T] \rightarrow \mathbb{R}^{d}$ be a centered, continuous Gaussian process with independent components. Assume that the covariance of every component has finite mixed $(1,\rho)$-variation for some $\rho \in [1,2)$ on $[0,T]^2$, that is, for $k=1,\dots,d$,
  %and uniformly over $s<t$ in $[0,T]$
\begin{equation*}
\sup_{\substack{(t_{i}), (t_{j}^{\prime}) \in\mathcal{D}([0,T])
}}
\left(\sum_{t'_{j}}\left(\sum_{t_{i}}\left|    E X^k_{t_{i},t_{i+1}}X^k_{t_{j}^{\prime},t_{j+1}^{\prime}}      
\right|\right)^{\rho}\right)^{\frac{1}{\rho}} \\ 
< \infty .
\end{equation*}%

\end{condition} 

We note that Condition \ref{mixed_p-var} is satisfied by large number of Gaussian examples \cite{FGGR16}, in particular
classical fBm, with Hurst parameter $H>1/4$, and many variants thereof. (The intuition behind this condition is
that $\rho$ measures the roughness of the covariance close to the diagonal, whereas $1$ deals with the off-diagonal
part and somehow expresses a ``sign" in the correlation structure (negative in case of fBm with $H<1/2$), while still 
allowing for sufficiently regular perturbations.)  There are two consequence of Condition \ref{mixed_p-var}  that will be most 
important to us. First, it guarantees the existence of a ``canonical" Gaussian rough path associated to $X$,
denoted by $\mathbf{X}$ (or $\mathbf{X}(\omega)$ if we want to stress its random nature). Secondly, it provides
regularity of the Cameron--Martin space $\mathcal{H}$ associated to $X$, 
\begin{equation}    \label{CMemb}
\iota \colon \mathcal{H}\hookrightarrow C^{q-\text{var}}   \text {  with }    q = \frac{1}{\frac{1}{2\rho} + \frac{1}{2}} < 2
\end{equation}%
which - cutting a long story (\cite{CLL13} or \cite[Ch.12]{FH14}) short - leads to probabilistic
estimates akin to those available within It\^o theory.

We will sometimes need a slightly stronger version of Condition \ref{mixed_p-var} which will ensure that our rough paths live in H\"older-type spaces:

\begin{condition}  \label{mixed_hoel} Let $X = (X^1,\ldots,X^d)\colon [0,T] \rightarrow \mathbb{R}^{d}$ be a centered, continuous Gaussian process with independent components. Assume that the covariance of every component has \emph{H\"older dominated} finite mixed $(1,\rho)$-variation for some $\rho \in [1,2)$ on $[0,T]^2$, that is, there exists $K<\infty$ such that, for $k=1,\dots,d$ and uniformly over $s<t$ in $[0,T]$,
%i.e. that there is some constant $C > 0$ such that for every $0\leq s < t \leq T$ and every $i = 1,\ldots,d$,
%\begin{align*}
% &\sup \left\{ \sum_{k,l} |E[X_{t_{k},t_{k+1}}^i X_{\hat{t}_{l},\hat{t}_{l+1}}^i]|^{\rho}\, \Big| \, s = t_0 <  \ldots < t_K = t,\ s = \hat{t}_0 <  \ldots < \hat{t}_L = t,\ K,L \in \N \right\} \\
% \leq\ &C |t-s|.
%\end{align*}
%Moreover, assume that complementary Young regularity holds for the trajectories of $X$ and the Cameron-Martin paths. 
%======
%that $R$ has (H%
%\"{o}lder controlled) mixed $\left( 1,\rho \right) $-variation, in the sense
%that 
\begin{equation*}
%V_{1,\rho }\left( R;\left[ s,t\right] ^{2}\right) & =
%XXXX
\sup_{\substack{(t_{i}), (t_{j}^{\prime}) \in\mathcal{D}([s,t])
}}
\left(\sum_{t'_{j}}\left(\sum_{t_{i}}\left|    E X^k_{t_{i},t_{i+1}}X^k_{t_{j}^{\prime},t_{j+1}^{\prime}}      
\right|\right)^{\rho}\right)^{\frac{1}{\rho}} \\ 
\le K \left( \left\vert
t-s\right\vert ^{1/\rho }\right) .
\end{equation*}%
\end{condition}

Note that fBm (with $\rho = \frac{1}{2H}$) and then any centered Gaussian process with 
 stationary increments s.t.  
$
\sigma ^{2}\left( t-s\right) :=E\left\vert X_{t}^{i}-X_{s}^{i}\right\vert
^{2}
$ is concave and $\sigma ^{2}\left( \tau \right) =\mathcal{O}\left(
\tau ^{1/\rho }\right) $ 
satisfies Condition \ref{mixed_hoel}, provided $\rho \in [1,2)$. This, and more examples, are discussed in \cite{FGGR16}.

\vspace{1mm}

In the following subsection, we will establish $L^p$-convergence rates for step-$N$ Euler approximations based on the entire Gaussian rough paths, i.e. schemes involving iterated (random) integrals up to order $N$. 
% Although these schemes are hard to implement (the distributions of the iterated integrals are in general not known) it will serve as a stepping stone to more simple schemes. 
We continue by giving $L^p$-rates for the Wong-Zakai theorem in the Gaussian case. Putting together both results, we can give $L^p$ convergence rates for an (easy-to-implement) simplified Euler scheme presented first in \cite{DNT12}. We will see that the (sharp) almost sure convergence rates obtained in \cite{FR14} also hold in $L^p$.

\subsection{~$L^{r}$-rates for step-$N$ Euler approximation (based on entire
rough path)}

For simplicity, the following Theorem is formulated only in the H\"{o}lder case.

\begin{theorem}\label{theorem_euler_lp} 

% {\color{yellow} **** REMOVE Assume that $V_{\rho }\left( R_{X};\left[ s,t\right] ^{2}\right)
% \leq K\left\vert t-s\right\vert ^{1/\rho }$ holds for all $s<t$, some $\rho
% \in \lbrack 1,2)$ and a constant $K$. Assume that 
% \begin{equation*}
% \iota \colon \mathcal{H}\hookrightarrow C^{q-\text{var}}
% \end{equation*}%
% for some $q < 2$, let $M\geq \|\iota\|_{\mathcal{H} \to C^{q-\text{var}}}$ and assume that complementary Young-regularity holds. ****}

Assume the driving Gaussian noise $X$ satisfies Condition \ref{mixed_hoel}. Choose $p>2\rho $, assume that $V\in Lip^{\theta }$ for some $\theta >p$ and let $\nu \geq
\left\vert V\right\vert _{Lip^{\theta }}$. Let $D$ be a dissection of $[0,T]$ with mesh size at most $h > 0$
% \begin{align*}
%   D=\left\{ 0<h
% <2h <\ldots <(\lfloor T/h \rfloor -1)h <T\right\}
% \end{align*}
and let $Y_{T}^{\text{Euler;}D}$ denote the step-$N$ Euler approximation of $Y$,
the (pathwise) solution of%
\begin{equation*}
dY=V(Y)\,d\mathbf{X~};\quad Y_{0}\in \mathbb{R}^{e}
\end{equation*}%
where $N$ is chosen such that $\lfloor p\rfloor \leq N\leq \theta $. 

Then for every $r\geq 1$, $r^{\prime }>r$ and $\zeta \in \left[ \frac{N}{p},\frac{%
N+1}{p}\right) $ there is a constant % $C=C(\rho ,p, q,\theta ,\nu,K,M,r,r^{\prime },N,\zeta )$ 
$C=C(\rho ,p, \theta ,\nu,K,r,r^{\prime },N,\zeta )$ 
such that%
\begin{equation*}
\left\vert Y_{T}-Y_{T}^{\text{Euler;}D}\right\vert _{L^{r}}\leq CT\left\vert
\left\Vert \mathbf{X}\right\Vert _{1/p\text{-H\"{o}l};\left[ 0,T\right]
}^{\zeta p}\right\vert _{L^{r^{\prime }}}h ^{\zeta -1}
\end{equation*}%
holds for all $h >0$.
\end{theorem}

\begin{remark}
By choosing $\hat{p}\in (2\rho ,p)$ one has $\frac{N+1}{p}<\frac{N+1}{\hat{p}%
}$ and applying Theorem \ref{theorem_euler_lp}  with $\hat{p}$ instead of $p$ shows that%
\begin{equation*}
\left\vert Y_{T}-Y_{T}^{\text{Euler;}D}\right\vert _{L^{r}}\leq C h
^{\frac{N+1}{p}-1}
\end{equation*}%
holds for every $p>2\rho $ if $h \rightarrow 0$.
\end{remark}

\begin{proof}[Proof of Theorem \ref{theorem_euler_lp}] 

The embedding (\ref{CMemb}) provides the so-called {\it complementary Young regularity} (see e.g. \cite[Sec. 11.1]{FH14}) meaning that integrals of the form $\int h dX, \int X dh$ are well-defined Young integrals. 
With this condition in place we can use (a variation of the theme of) \cite{CLL13} to conclude sharp tail estimates for $N_{\alpha ,[0,T]}(\mathbf{X})$. More precisely, choosing $q$ as in \eqref{CMemb}, \cite[Lemma 5 and Corollary 2]{FR12b} show that there is
an $\alpha =\alpha (p,\rho ,K)>0$ and a positive constant $c_{1}=c_{1}\left(
p,\rho ,K\right) $ such that 
\begin{equation*}
P(N_{\alpha ,[0,T]}(\mathbf{X})>u)\leq \exp \left\{ -c_{1}\alpha 
^{2/p}u^{2/q}\right\}
\end{equation*}%
holds for all $u>0$. Now we use the pathwise estimate $\eqref{eqn_euler_estim_hoelder_case}$ and take the $L^{r}$ norm on both sides of the inequality. The H\"{o}lder inequality
shows that%
\begin{equation*}
\left\vert Y_{T}-Y_{T}^{\text{Euler;}D}\right\vert _{L^{r}}\leq
c_{1}T\left\vert \left\Vert \mathbf{X}\right\Vert _{1/p\text{-H\"{o}l};\left[
0,T\right] }^{\zeta p}\right\vert _{L^{r^{\prime }}}\left\vert \exp \left\{
C\left( N_{\alpha ,[0,T]}(\mathbf{X})+1\right) \right\} \right\vert
_{L^{r^{\prime \prime }}}\left\vert D\right\vert ^{\zeta -1}
\end{equation*}%
holds for some (possibly large) $r^{\prime \prime }>r$. Our tail estimate for $N_{\alpha ,[0,T]}(\mathbf{X})$ shows that the $L^{r''}$ norm of the exponential term above is finite which yields the claim.
\end{proof}

\subsection{$L^{r}$-rates for Wong-Zakai approximations}

\label{lr_rates_wong_zakai}

We aim to formulate a version of the Wong-Zakai Theorem which contains
convergence rates in $L^{r}$, any $r\geq 1$ for a class of suitable
approximations $X^{h }$ of $X$. By this, we mean that $X^h$ is a centered, continuous Gaussian process with independent components for every $h \in (0,1]$ and that
\begin{enumerate}
\item[(i)] $\left( X^{h },X\right) \colon \left[ 0,T\right] \rightarrow 
\mathbb{R}^{d+d}$ is jointly Gaussian, $\left( X^{h ;i},X^{i}\right) $
and $\left( X^{h ;j},X^{j}\right) $ are independent for $i\neq j$ and%

\begin{equation}\label{eqn:unif_rho-var}
  \sup_{h \in (0,1]} V_{1,\rho } ( R_{( X^{h },X) }; 
[0,T] ^{2}) =:K < \infty
\end{equation}%

% 
% {\color{yellow}
% \begin{equation*} %\label{eqn:unif_rho-var}
% \sup_{h \in (0,1]}V_{\rho }\left( R_{\left( X^{h },X\right) }; 
% \left[ 0,T\right] ^{2}\right) =:K<\infty
% \end{equation*}}
for $\rho \in \lbrack 1,2)$ as in Condition \ref{mixed_p-var}. 

(Note that this implies that Condition \ref{mixed_p-var} also holds for every $X^h$, $h \in  (0,1]$.)

\item[(ii)] Uniform convergence of the second moments:%
\begin{equation*}
\sup_{t\in \left[ 0,T\right] }E\left[ \left\vert X_{t}^{h
}-X_{t}\right\vert ^{2}\right] =:\delta \left( h \right) ^{1/\rho
}\rightarrow 0\quad \text{for }h \rightarrow 0.
\end{equation*}
\end{enumerate}

% {\color{yellow}
% \begin{example}
% A typical example of such approximations are the piecewise linear
% approximations of $X\left( \omega \right) $ at the time points $\left\{
% 0<h <2h <\ldots <(\lfloor T/h \rfloor -1)h
% <T\right\} $ (see \cite[Chapter 15.5]{FV10}). In the case $V_{\rho }\left(
% R_{X};\left[ s,t\right] ^{2}\right) \lesssim \left\vert t-s\right\vert
% ^{1/\rho }$ (i.e. if we deal with H\"{o}lder rough paths), one can show that 
% $\delta \left( h \right) \lesssim h $.
% \end{example}
% }
\begin{theorem}\label{theorem_wong_zakai_lp}
% *** FIX *** Let $X\colon \left[ 0,T\right] \rightarrow 
% \mathbb{R}^{d}$ be a centered Gaussian process with continuous sample paths, independent components and covariance of finite $\rho$-variation for some $\rho \in [1,2)$. 

Assume the driving Gaussian noise $X$ satisfies Condition \ref{mixed_p-var} and let $(X^{h })_{h > 0}$ be a family of suitable approximations as above.
% Let $\mathcal{H}^{h }$ and $\mathcal{H}^0$
% denote the Cameron-Martin spaces of the processes $X^{h}$ resp. $X$. Assume that 
% \begin{equation*}
% \iota^{h} \colon \mathcal{H}^{h}\hookrightarrow C^{q-\text{var}}
% \end{equation*}%
% for some $q < 2$, let $M\geq \|\iota^{h}\|_{\mathcal{H} \to C^{q-\text{var}}}$ for all $h > 0$ and assume that complementary Young-regularity holds.
 Let $\mathbf{X}$ and $\mathbf{X}^{h }$denote the lift of $X$ resp. $%
X^{h }$ to a process with $p$-rough sample paths for some $p>2\rho $.
Let $V=(V_{1},\ldots ,V_{d})$ be a collection of vector fields in $\mathbb{R}%
^{e}$. Choose $\eta <\frac{1}{\rho }-\frac{1}{2}$ and assume that  $\left\vert V\right\vert_{Lip^{\theta }}\leq \nu <\infty $ for some
$\theta > \frac{2\rho }{1 - 2\rho \eta}$. Let $Y,\,Y^{h }\colon \left[ 0,T%
\right] \rightarrow \mathbb{R}^{e}$ denote the pathwise solutions to the
equations 
\begin{eqnarray*}
dY_{t} &=&V(Y_{t})\,d\mathbf{X}_{t};\quad Y_{0}\in \mathbb{R}^{e} \\
dY_{t}^{h } &=&V(Y_{t}^{h })\,d\mathbf{X}_{t}^{h
};\quad Y_{0}^{h }=Y_{0}\in \mathbb{R}^{e}.
\end{eqnarray*}%
Then, for any $r\geq 1$ there is a
constant $C=C(\rho ,p,\theta ,\nu ,K,\eta ,r)$ such that%
\begin{equation*}
\left\vert \left\vert Y^{h }-Y\right\vert _{\infty ;\left[ 0,T\right]
}\right\vert _{L^{r}}\leq C\delta \left( h \right) ^{\eta }
\end{equation*}%
holds for all $h >0$.
\end{theorem}

A typical example of such approximations are the \emph{piecewise linear
approximations}:

\begin{corollary}\label{cor:wong_zakai_pl} Assume the driving Gaussian noise $X$ satisfies Condition \ref{mixed_hoel} and that $X^h$ is a piecewise linear approximation of $X$ with mesh-size at most $h$. Then
\begin{equation*}
\left\vert \left\vert Y^{h }-Y\right\vert _{\infty ;\left[ 0,T\right]
}\right\vert _{L^{r}}\leq C h^{\eta }
\end{equation*}%
for any $\eta < \frac{1}{\rho }-\frac{1}{2}$
\end{corollary}

\begin{proof}[Proof of Corollary \ref{cor:wong_zakai_pl}] We need to check that piecewise linear approximations are ``suitable'' in the sense of the beginning of this section. Secondly, we need to see
that $E  \left\vert X_{t}^{h}-X_{t}\right\vert ^{2}   \leq C h^{1/ \rho}$
uniformly in $t\in[0,T]$. To simplify notation, we will assume $d = 1$. Let $D = \left\{0 < s_1 < \ldots < s_M = T \right\}$ be a dissection of $[0,T]$ such that $|s_{k+1} - s_k| \leq h$, $k = 0,\ldots,M-1$, and let $X^h$ denote the piecewise linear approximation of $X$ at the time points given by $D$. Concerning the first point, we have to show that $V_{1,\rho } ( R_{( X^{h },X) }; [0,T] ^{2}) $ is uniformly bounded in $h$.  Let $D_1$ and $D_2$ be two arbitrary dissections of $[0,T]$. Set $\bar{D}_1 := D_1 \cup D$. By the triangle inequality,
\begin{align*}
 \sum_{t'_j \in D_2} \left( \sum_{t_i \in D_1} \left| E X^h_{t_i,t_{i+1}} X_{t'_j,t'_{j+1}} \right| \right)^{\rho} &\leq \sum_{t'_j \in D_2} \left( \sum_{\bar{t}_i \in \bar{D}_1} \left| E X^h_{\bar{t}_i,\bar{t}_{i+1}} X_{t'_j,t'_{j+1}} \right| \right)^{\rho} \\
 &= \sum_{t'_j \in D_2} \left( \sum_{s_k \in D} \left| E X_{s_k,s_{k+1}} X_{t'_j,t'_{j+1}} \right| \right)^{\rho} \\
 &\leq V_{1,\rho } ( R_{X }; [0,T] ^{2})^{\rho}.
\end{align*}
  Using the basic estimate $(a + b)^{\rho} \leq 2^{\rho-1} (a^{\rho} + b^{\rho})$ instead of the triangle inequality, we similarly obtain
\begin{align*}
 \sum_{t'_j \in D_2} \left( \sum_{t_i \in D_1} \left| E X_{t_i,t_{i+1}} X^h_{t'_j,t'_{j+1}} \right| \right)^{\rho} &\leq 2^{\rho-1} V_{1,\rho } ( R_{X }; [0,T] ^{2})^{\rho}
\end{align*}
  and
  \begin{align*}
 \sum_{t'_j \in D_2} \left( \sum_{t_i \in D_1} \left| E X^h_{t_i,t_{i+1}} X^h_{t'_j,t'_{j+1}} \right| \right)^{\rho} &\leq 2^{\rho-1} V_{1,\rho } ( R_{X }; [0,T] ^{2})^{\rho}.
\end{align*}
Taking the supremum over all dissections, these estimates imply that $V_{1,\rho } ( R_{( X^{h },X) }; [0,T] ^{2})$ is bounded from above by $V_{1,\rho } ( R_{X }; [0,T] ^{2})$ times a constant which only depends on $\rho$. Concerning the second point, note that for $t \in [s_k,s_{k+1}]$,
\begin{align*}
  |X^h_t - X_t| \leq |X_{s_{k+1}} - X_{s_k}| + |X_{t} - X_{s_k}|
\end{align*}
and therefore, using Condition \ref{mixed_hoel},
\begin{align*}
 E| X_{t}^{h}-X_{t}|^{2} &\leq 2 E|X_{s_{k+1}} - X_{s_k}|^2 +2 E |X_{t} - X_{s_k}|^2 \leq 4 V_{1,\rho } ( R_{X }; [s_k,s_{k+1}] ^{2}) \leq 4 K |s_{k+1} - s_k|^{1/\rho} \\ 
 &\leq 4K h^{1/ \rho}
\end{align*}
which implies the second point.

\end{proof}

% {\color{yellow}
% 
% \begin{remark}
% We give some sufficient conditions under which the assumptions on the Cameron-Martin paths in Theorem \ref{theorem_wong_zakai_lp} are fulfilled:
% \begin{enumerate}
%  \item If the Cameron-Martin paths associated to the process $X$ have finite $q$-variation, if complementary Young-regularity holds for the trajectories of $X$ and the Cameron-Martin paths and if the operators $\Lambda_{h} \colon \omega \mapsto \omega^{h}$ are uniformly bounded in the sense that
%  \begin{align*}
%   \sup_{h > 0} \| \Lambda_{h} \|_{C^{q-\text{var}} \to C^{q-\text{var}}} < \infty,
%  \end{align*}
%  we have $\sup_{h > 0} \| \iota^{h} \|_{\mathcal{H} \to C^{q-\text{var}}} < \infty$. This is the case, for instance, when dealing with piecewise-linear or mollifier approximations.
%  \item In the case $\rho \in \lbrack 1,3/2)$, the assumptions are always fulfilled using \eqref{eqn:unif_rho-var} and the Cameron-Martin embedding from \cite[Proposition 17]{FV10-2} and we can set $M = \sqrt{K}$.
%  \item For $\rho \in [1,2)$, another sufficient condition is uniform mixed $(1,\rho)$-variation:
% \begin{equation*}
%   \sup_{h \in (0,1]} V_{1,\rho } ( R_{( X^{h },X) }; 
% [0,T] ^{2}) =:K' < \infty.
% \end{equation*}%
% In this case, we can choose $K = K'$ and $M =  \sqrt{K'}$ {\color{red}{(see \cite[Theorem 1]{FGGR15})}}. This holds, for instance, for fractional Brownian motion with Hurst parameter $H > 1/4$.
% \end{enumerate}
% 
% \end{remark}
% }

\begin{proof}[Proof of Theorem \ref{theorem_wong_zakai_lp}]
Set $X^0 := X$ and $\mathbf{X}^0 := \mathbf{X}$. Let $\mathcal{H}^h$ denote the Cameron--Martin space associated to $X^h$, $h \geq 0$. Using the uniform bound \eqref{eqn:unif_rho-var}, \cite[Theorem 1]{FGGR16} implies that
\begin{equation*}
\left\vert \phi\right\vert _{q-\text{var}}\leq \sqrt{K} \left\vert \phi \right\vert _{%
\mathcal{H}^{h }}
\end{equation*}%
holds for every $\phi \in \mathcal{H}^{h }$ and $h \geq 0$ with $q$ as in \eqref{CMemb}. As in the proof of Theorem \ref{theorem_euler_lp}, we can find an $\alpha =\alpha (p,\rho ,K)>0$ and a positive constant $c_{1}=c_{1}\left(p,\rho ,K\right) $ such that the uniform tail estimate%
\begin{equation*}
P(N_{\alpha ,[0,T]}(\mathbf{X}^{h })>u)\leq \exp \left\{ -c_{1}\alpha
^{2/p}u^{2/q}\right\} \quad \text{for all } u >0
\end{equation*}%
holds for all $h \geq 0$. Choose $\hat{p}\in \left( \frac{2\rho }{1 - 2\rho \eta}, \theta \right)$
and set $\mathbf{\hat{X}}^{h
}=S_{\lfloor \hat{p}\rfloor }\left( \mathbf{X}^{h }\right) $ for $h \geq 0$. Lipschitzness of
the map $S_{\lfloor \hat{p}\rfloor }$ and \cite[Lemma 2]{FR12b} show that
also%
\begin{equation}
P(N_{\alpha ,[0,T]}(\mathbf{\hat{X}}^{h })>u)\leq \exp \left\{
-c_{1}\alpha ^{2/p}u^{2/q}\right\}  \quad \text{for all } u >0 \label{eqn_unif_tails_rates}
\end{equation}%
holds for all $h \geq 0$ for a possibly smaller $\alpha >0$ (depending on $\hat{p}$).
Now we use Theorem \ref{cor_loc_lip_integrability_unif_top} and the
Cauchy-Schwarz inequality to see that%
\begin{equation*}
\left\vert \left\vert Y^{h }-Y\right\vert _{\infty ;\left[ 0,T\right]
}\right\vert _{L^{r}}\leq c_{2}\left\vert \rho _{\hat{p}-\text{var};[0,T]}(%
\mathbf{\hat{X}}^{h },\mathbf{\hat{X}})\right\vert
_{L^{2r}}\left\vert \exp \left\{ c_{2}\left( N_{\alpha ,[0,T]}(\mathbf{\hat{X%
}}^{h })+N_{\alpha ,[0,T]}(\mathbf{\hat{X}})+1\right) \right\}
\right\vert _{L^{2r}}
\end{equation*}%
holds for a constant $c_{2} > 0$. The uniform tail estimates $\left( \ref%
{eqn_unif_tails_rates}\right) $ show that%
\begin{equation*}
\sup_{h \geq 0}\left\vert \exp \left\{ c_{2}\left( N_{\alpha ,[0,T]}(%
\mathbf{\hat{X}}^{h })+N_{\alpha ,[0,T]}(\mathbf{\hat{X}})+1\right)
\right\} \right\vert _{L^{2r}}\leq c_{3}<\infty .
\end{equation*}%
Applying \cite[Theorem 5]{FR14} with $\gamma  = \frac{\rho }{1 - 2\rho \eta}$ shows that
\begin{equation*}
\left\vert \rho _{\hat{p}-\text{var};[0,T]}(\mathbf{\hat{X}}^{h },%
\mathbf{\hat{X}})\right\vert _{L^{2r}}\leq c_{4}\sup_{t\in \left[ 0,T\right]
}\left\vert X_{t}^{h }-X_{t}\right\vert _{L^{2}}^{1-\frac{\rho }{%
\gamma }}= c_4 \delta \left( h \right) ^{\eta }
\end{equation*}%
for a constant $c_{4}$ which yields the claim.
\end{proof}

\subsection{$L^{r}$-rates for the simplified Euler schemes}\label{sec:prob-conv-results}

For $N\geq 2$, step-$N$ Euler schemes contain iterated integrals whose
distributions are not easy to simulate when dealing with Gaussian processes.
In contrast, the simplified step-$N$ Euler schemes avoid this difficulty by
substituting the iterated integrals by a product of increments. In the
context of fractional Brownian motion, it was introduced in \cite{DNT12}. We
make the following definition: If $V=\left( V_{1},\ldots ,V_{d}\right) $ is
sufficiently smooth, $\mathbf{x}$ is a $p$-rough path, $y\in \mathbb{R}^{e}$
and $N\geq \lfloor p\rfloor $, we set%
\begin{equation*}
\mathcal{E}_{\left( V\right) }^{\text{simple}}\left( y,S_{N}\left( \mathbf{x}%
\right) _{s,t}\right) :=\sum_{k=1}^{N}\frac{1}{k!}V_{i_{1}}\ldots
V_{i_{k}}I\left( y\right) x_{s,t}^{i_{1}}\cdots x_{s,t}^{i_{k}}
\end{equation*}%
for $s<t$ and%
\begin{equation*}
\mathfrak{E}_{\text{simple}}^{S_{N}\left( \mathbf{x}\right) _{s,t}}y:=y+%
\mathcal{E}_{\left( V\right) }^{\text{simple}}\left( y,S_{N}\left( \mathbf{x}%
\right) _{s,t}\right) .
\end{equation*}%
Given $D=\left\{ 0=t_{0}<\ldots <t_{n}=T\right\} $ and a path $\mathbf{x}\in
C^{p-var}_0 \left( \left[ 0,T\right] ;G^{\lfloor p\rfloor }\left( \mathbb{R}%
^{d}\right) \right) $ we define the \emph{simplified} (step-$N$) Euler
approximation to the RDE solution $y$ of%
\begin{equation*}
dy=V\left( y\right) \,d\mathbf{x}
\end{equation*}%
with starting point $y_{0}\in \mathbb{R}^{e}$ at time $t_{k}\in D$ by%
\begin{equation*}
y_{t_{k}}^{\text{simple Euler;}D}:=\mathfrak{E}_{\text{simple}%
}^{t_{k}\leftarrow t_{0}}y_{0}:=\mathfrak{E}_{\text{simple}}^{S_{N}\left( 
\mathbf{x}\right) _{t_{k-1},t_{k}}}\circ \cdots \circ \mathfrak{E}_{\text{%
simple}}^{S_{N}\left( \mathbf{x}\right) _{t_{0},t_{1}}}y_{0}
\end{equation*}
and at time $t \in (t_k,t_{k+1})$ by

\begin{align*}
y_{t}^{\text{simple Euler;}D} := \left( \frac{t - t_k}{t_{k+1} - t_k}
\right) \left( y_{t_{k+1}}^{\text{simple Euler;}D} - y_{t_k}^{\text{simple
Euler;}D} \right) + y_{t_k}^{\text{simple Euler;}D}.
\end{align*}

\begin{theorem}\label{cor:simplified-stepN-euler}
Assume the driving Gaussian noise $X$ satisfies Condition \ref{mixed_hoel}. 
Choose $N \in \{2,3\}$, $\eta_1 > 0$ and $\eta_2 > 0$ such that 
\begin{align*}
  N > 2\rho - 1, \quad \eta_1 < \frac{1}{\rho} - \frac{1}{2} \quad \text{and} \quad
  %,\quad \eta_2 < \frac{1}{2\rho}
  %\quad\text{and}\quad 
  \eta_2 < \frac{N + 1}{2\rho} - 1.
\end{align*}
 Assume that $\left\vert V\right\vert _{Lip^{\theta }}\leq \nu <\infty $ for some $\theta \in (2,\infty ]$ which satisfies
  $\theta > \frac{2\rho }{1 - 2\rho\eta_1}$ and $\theta \geq N$. Let $D$ be a dissection of $[0,T]$ with mesh size at most $h > 0$.
%  Set
% \begin{align*}
%   D=\left\{ 0<h <2h <\ldots <(\lfloor T/h
% \rfloor -1)h <T\right\}
% \end{align*}
% for $h >0$. 

Then for any $r\geq 1$ there is a constant $C=C(\rho, K, N, \eta_1, \eta_2, \theta ,\nu ,r)$ such that
\begin{equation*}
\left| \left\vert Y - Y^{\text{simple Euler;}D}\right\vert_{\infty}
\right|_{L^r} \leq C ( h^{\eta_1} + h^{\eta_2}) 
% + h ^{%
% \frac{1}{\rho }-\frac{1}{2}-\delta }+h ^{\frac{N+1}{2\rho }-1-\delta }
\end{equation*}%
for all $h >0$.
\end{theorem}

\begin{remark}
In the proof, we will see that the rate $\eta_1$ is the rate
for the Wong-Zakai approximation
%, the rate $\eta_2$ comes from the (almost) $\frac{1}{2\rho}$-H\"{o}lder-regulartiy of the sample paths of $Y$ 
and $\eta_2$ comes from the rate of the step-$N$ Euler approximation. 
%Since $\rho \geq 1$, the Wong-Zakai error always dominates $\eta_2$. 
In particular, for $%
\rho =1$, we can choose $N=2$ to obtain a rate arbitrary close to $\frac{1}{2}
$ and the rate does not increase even if we choose $N = 3$. For $\rho >1$, the choice $N=3$ gives a rate of almost $\frac{1}{\rho }-%
\frac{1}{2}$. 
%In both cases the rate does not increase for larger choices of $N$.
\end{remark}

From this remark, we immediately obtain

\begin{corollary}\label{cor:strong_rates_smooth_vf}
Assume the driving Gaussian noise $X$ satisfies Condition \ref{mixed_hoel}. 
 Assume that the vector fields $V = (V_1,\ldots,V_d)$ are bounded, $C^\infty$ with bounded derivatives.

 (i) Case of $\rho < 3/2$. The simplified step-2 Euler scheme converges in $L^r$, for any $r\geq 1$, and rate $\frac{3}{2 \rho} - 1 - \delta$, for any $\delta > 0$.
 
 (ii) In the general case of  $\rho < 2$, the simplified step-3 Euler scheme converges in $L^r$, for any $r\geq 1$, and rate $\frac{1}{\rho} - \frac{1}{2} - \delta$, for any $\delta > 0$, 

% Then the simplified step-3 Euler scheme (step-2 in the case $\rho = 1$) converges in $L^r$, for any $r\geq 1$, and rate $\frac{1}{\rho} - \frac{1}{2} - \delta$, for any $\delta > 0$, to the solution of the corresponding rough differential equation. % {\color{red} NEED SOME HOELDER CONTROL HERE}
\end{corollary}

\begin{proof}[Proof of Theorem \ref{cor:simplified-stepN-euler}] 
Let $X^{h }$ denote the Gaussian process whose sample paths are
piecewise linear approximated at the time points given by $D$ and let $%
Y^{h }\colon \left[ 0,T\right] \rightarrow \mathbb{R}^{e}$ denote the
pathwise solution to the equation 
\begin{equation*}
dY^{h }=V(Y^{h })\,dX^{h };\quad Y_{0}^{h
}=Y_{0}\in \mathbb{R}^{e}.
\end{equation*}%
Then for any $t_{k},t_{k+1}\in D$ we have%
\begin{equation*}
\mathbf{X}_{t_{k},t_{k+1}}^{h ;k;i_{1},\ldots ,i_{k}}=\frac{1}{k!}%
X_{t_{k},t_{k+1}}^{i_{1}}\cdots X_{t_{k},t_{k+1}}^{i_{k}},
\end{equation*}%
hence $Y_{t}^{\text{simple Euler;}D}=Y_{t}^{h ;\text{ Euler;}D}$ for
any $t\in D$ and thus%
\begin{equation*}
\left\vert Y_{t}-Y_{t}^{\text{simple Euler;}D}\right\vert \leq \left\vert
Y-Y^{h }\right\vert _{\infty }+\max_{t_{k}\in D}\left\vert
Y_{t_{k}}^{h }-Y_{t_{k}}^{h ;\text{ Euler;}D}\right\vert
\end{equation*}%
if $t\in D$. For $t\notin D$, choose $t_{k}\in D$ such that $t_{k}<t<t_{k+1}$%
. Set $a=\frac{t-t_{k}}{t_{k+1}-t_{k}}$ and $b=\frac{t_{k+1}-t}{t_{k+1}-t_{k}%
}$, i.e. $a+b=1$. In the following, the relation $A \lesssim B$ means $A \leq \text{const.}\, B$ where the constant does not depend on $h$ or $t$. By the triangle inequality,%
\begin{eqnarray*}
\left\vert Y_{t}-Y_{t}^{\text{simple Euler;}D}\right\vert &\leq &a\left\vert
Y_{t}-Y_{t_{k+1}}\right\vert +b\left\vert Y_{t}-Y_{t_{k}}\right\vert
+a\left\vert Y_{t_{k+1}}-Y_{t_{k+1}}^{\text{simple Euler;}D}\right\vert \\
 &\quad &+\ b\left\vert Y_{t_{k}}-Y_{t_{k}}^{\text{simple Euler;}D}\right\vert \\
&\lesssim &h ^{1/p}\left\Vert Y\right\Vert _{1/p\text{-H\"{o}l};\left[
0,T\right] }+\max_{t_{k}\in D}\left\vert Y_{t_{k}}-Y_{t_{k}}^{\text{simple
Euler;}D}\right\vert \\
&\lesssim &h ^{1/p}\left( \left\Vert \mathbf{X}\right\Vert _{1/p\text{%
-H\"{o}l};\left[ 0,T\right] }\vee \left\Vert \mathbf{X}\right\Vert _{1/p%
\text{-H\"{o}l};\left[ 0,T\right] }^{p}\right) +\left\vert Y-Y^{h
}\right\vert _{\infty } \\
&\quad &+\ \max_{t_{k}\in D}\left\vert Y_{t_{k}}^{h
}-Y_{t_{k}}^{h ;\text{ Euler;}D}\right\vert
\end{eqnarray*}%
for $p>2\rho $ sufficiently small, where we used \cite[Theorem 10.14]{FV10} in the last inequality. Since the estimate holds uniformly over $t$, we can
pass to the sup-norm on the left hand side of the inequality. We now take the $L^r$-norm and use the triangle inequality on the right hand side.

Since $\mathbf{X}$ is the lift of a Gaussian process, $\left\Vert \mathbf{X}\right\Vert _{1/p\text{-H\"{o}l};\left[ 0,T\right]}$ has Gaussian tails (\cite[Theorem 15.33]{FV10}). Therefore, all its moments are finite, and we can choose $p$ such that $1/\eta_1 \geq p > 2\rho$ to obtain
\begin{align*}
 \left| h ^{1/p}\left( \left\Vert \mathbf{X}\right\Vert _{1/p\text{%
-H\"{o}l};\left[ 0,T\right] }\vee \left\Vert \mathbf{X}\right\Vert _{1/p%
\text{-H\"{o}l};\left[ 0,T\right] }^{p}\right) \right|_{L^r} \lesssim h^{\eta_1}.
\end{align*}

Corollary \ref{cor:wong_zakai_pl} implies that
\begin{equation*}
\left| \left\vert Y -Y^{h }\right\vert_{\infty} \right|_{L^r} \lesssim h ^{\eta_1}
\end{equation*}%
holds for all $h >0$. Now we choose $p'>2\rho $
such that $\frac{N+1}{p'}-1= \eta_2 $ and apply
Theorem \ref{theorem_euler_lp} to estimate the last term. Since $\left| d_{1/p'\text{-H\"{o}l}}(\mathbf{X}^{h}, \mathbf{X})\right|_{L^r} \to 0$ for $h \to 0$, clearly $\sup_{h>0} \left| \left\Vert \mathbf{X}^{h
}\right\Vert _{1/p'\text{-H\"{o}l};\left[ 0,T\right] } \right|_{L^r} < \infty$ and we obtain a uniform estimate of the form
\begin{align*}
 \left| \max_{t_{k}\in D}\left\vert Y_{t_{k}}^{h
}-Y_{t_{k}}^{h ;\text{ Euler;}D}\right\vert \right|_{L^r} \lesssim h^{\eta_2}
\end{align*}
which yields the claim.
% 
% 
% . Choosing $p$ such that $\frac{1}{p} = \eta_2$ gives the claim.
\end{proof}

%%% Local Variables: 
%%% mode: latex
%%% TeX-master: "MLMC_RP_revision_final"
%%% End: 

%\input{BFRS_Teil2_final}

\section{Multilevel simulation of RDEs}
\label{sec:mult-simul-rdes}

In the spirit of Giles~\cite{G08b} we consider a multilevel Monte Carlo
procedure in connection with the developed schemes for RDEs. In this context
we reconsider and refine the complexity analysis by Giles~\cite{G08b} in certain
respects. On the one hand we relax the requirement $\alpha\ge1/2$ in
Giles~\cite{G08b} concerning the bias rate, and on the other we keep track of
various proportionality constants more carefully. 
  M\"uller-Gronbach and Ritter~\cite[Theorem 1]{MGK09} give a very general
  abstract multilevel Monte Carlo complexity result, which includes
  Theorem~\ref{thr:ml-complexity} (but not
  Theorem~\ref{thr:ml-complexity-log}) as a special case as far as rates are
  concerned. However, we also feel that the balance between the various
  constants of proportionality involved can make a big difference for the
  performance of a multilevel algorithm in practice.
 
(Cf.~the importance of various proportionality constants in the multilevel
Andersen-Broadie algorithm for simulating dual prices of American options due
to multilevel sub-simulation in~\cite{BSD13}.  See also Collier
  et al.~\cite{CHNST15} for an empirical approach to constructing optimal
  multilevel Monte Carlo algorithms.)  Furthermore, we will also give a
  discussion about the optimal balance between bias and variance in the
  multilevel Monte Carlo algorithm in Section~\ref{sec:balanc-bias-vari} below.

We adapt the main theorem of \cite{G08b} to our needs. Below one
should think
\begin{equation*}
P=f\left( Y_{\cdot}\right) 
\end{equation*}
for a Lipschitz function $f$ and $Y$ the solution to the Gaussian RDE $%
dY=V\left( Y\right) dX$. Let $\widehat{P}_{l}$ denote some (modified)
Milstein approximation \`a la \cite{DNT12}, for instance~%
\eqref{eq:simplified-stepN-euler}, based on a meshsize $h_{l}=T/(M_{0}M^{l}).
$ Recall the basic idea 
\begin{align*}
E\left[ P\right] & \approx E\left[ \widehat{P}_{L}\right] \text{ for }L\text{
large} \\
& =E\left[ \widehat{P}_{0}\right] +\sum_{l=1}^{L}E\left[ \widehat{P}_{l}-%
\widehat{P}_{l-1}\right]
\end{align*}
set $\widehat{P}_{-1}\equiv0$ and define the (unbiased) estimator $\widehat {%
Y}_{l}$ of $E\left[ \widehat{P}_{l}-\widehat{P}_{l-1}\right] $, say 
\begin{equation}
\widehat{Y}_{l}=\frac{1}{N_{l}}\sum_{i=1}^{N_{l}}\left( \widehat{P}%
_{l}^{\left( i\right) }-\widehat{P}_{l-1}^{\left( i\right) }\right) 
\label{defYhatl}
\end{equation}
based on $i=1,\dots,N_{l}$ independent samples. Note that $\widehat{P}%
_{l}^{\left( i\right) }-\widehat{P}_{l-1}^{\left( i\right) }$ comes from
approximations with different mesh but the same realization of the driving
noise.

    Any implementation of our proposed algorithm relies on samples of the
    increments of the underlying Gaussian process $X$, say on a grid with size
    $h_l^{-1}$. In the following, we assume that we know the covariances of
    those increments in closed form. For concreteness, let $\Sigma$ be the
    covariance matrix of the vector $\Delta X$ (of size $h_l^{-1}$) of
    increments of $X$. Clearly, we can always obtain samples from $\Delta X$
    by the Cholesky factorization of $\Sigma$, at cost proportional to
    $h_l^{-2}$---disregarding the one-off cost of computing the Cholesky
    factorization itself. In this case, the cost of simulating one trajectory
    of the approximate solution at level $l$ is, thus, proportional to
    $h_l^{-2}$. 

    On the other hand, when the increments are stationary,
    $\Sigma$ can be embedded into a circulant matrix, and FFT-methods can be
    employed to sample $\Delta X$ at cost proportional to $h_l^{-1}
    \log(h_l^{-1})$. (This case includes the fractional Brownian motion.) We
    refer to Dieker~\cite{D04} for a description of this and other related
    methods.

    Finally, in case of a standard Brownian motion, it is of course possible
    to simulate $\Delta X$ at cost proportional to $h_l^{-1}$. For the
    multilevel analysis below, all three cases are going to be addressed.

\subsection{Giles' complexity theorem revisited}

\begin{theorem}
  \label{thr:ml-complexity}
  
  In the spirit of Giles, we assume that there are constants $c_{1}$,
  $c_{2}^{\prime}$, $c_{2}$, $c_3$ and a rate $\gamma$ such that
  \begin{itemize}
  \item[(i)] $E\left[ \widehat{P}_{l} - P \right] \le c_{1} h_{l}^{\alpha}$,
  \item[(ii)] $E\left[ \widehat{Y}_{0}\right] = E\left[ \widehat{P}_{0}\right] 
    $ and $E\left[ \widehat{Y}_{l}\right] = E\left[ \widehat {P}_{l} - \widehat{P%
      }_{l-1}\right] $, $l>0$,
  \item[(iii)] $\var\left[ \widehat{Y}_{0}\right] \leq c_{2}^{\prime}N_{0}^{-1}$
    and $\var\left[ \widehat{Y}_{l}\right] \leq c_{2}N_{l}^{-1}h_{l}^{\beta}$ for $%
    l\in\mathbb{N}$,\footnote{%
      We distinguish between $c_{2}^{\prime}$ and $c_{2}$, since the former
      controls the variance $\var\left[ \widehat{Y}_{0}\right] $, which is often
      already proportional to the variance of $f(Y_{\cdot})$, whereas the latter
      controls the variance of the \emph{difference} $\widehat{Y}_{l}$, which is
      often much smaller in size.}
  \item[(iv)] $C_l \le c_3 N_l h_l^{-\gamma}$, $l \ge 0$,
  \end{itemize}
  where $C_l$ denotes the computational cost at level $l$.
  We further need to assume that $0 < \beta < \gamma$, $0<\alpha$.

Then for every $\varepsilon>0$, there are choices $L$ and $N_{l}$, $0\leq l\leq
L$, to be given below in~\eqref{eq:Loptimal} and~\eqref{eq:Nloptimal},
respectively, and constants $c_{4}$ and $c_{5}$ given in~%
\eqref{eq:complexity-bound} together with~\eqref{eq:ml-constants} such that
the multilevel estimator $\widehat{Y} =\sum_{l=0}^{L}\widehat{Y}_{l}$
satisfies the mean square error bound 
\begin{equation*}
\operatorname{MSE}\equiv E\left[ \left( \widehat{Y}-E[P]\right) ^{2}\right]
\leq \varepsilon^{2}, 
\end{equation*}
with complexity bound 
\begin{equation*}
C\leq
\begin{cases}
  c_4 \varepsilon^{-\frac{\gamma+2\alpha-\beta}{\alpha}} + o\left(
    \varepsilon^{-\frac{\gamma+2\alpha-\beta}{\alpha}} \right), & 2 \alpha >
  \beta,\\
  (c_4+c_5) \varepsilon^{-\frac{\gamma+2\alpha-\beta}{\alpha}} + o\left(
    \varepsilon^{-\frac{\gamma+2\alpha-\beta}{\alpha}} \right), & 2 \alpha =
  \beta,\\
  c_5 \varepsilon^{-\gamma/\alpha} + o\left( \varepsilon^{-\gamma/\alpha}
  \right), & 2 \alpha < \beta.
\end{cases}
\end{equation*}
\end{theorem}

\begin{proof}
The basic structure of the proof is closely based on the corresponding proof
of Giles~\cite{G08b}. Hence, we will not give all the details. Note that the
parameters $T$ and $M_0$ only enter into the picture in the form
$T/M_0$. Without loss of generality, we may therefore set $M_0=1$.

 As typical, the
first step consists in a standard Lagrangian optimization procedure
(minimizing the complexity constraint by the MSE), where one ignores the
requirement of $L$ and $N_l$ being integers. In the second step one then
chooses integer valued parameters that are close to the optimal real-valued
ones. 

The mean-square-error satisfies
\begin{equation*}
  \operatorname{MSE} =E\left[  \left(  \widehat{Y}-E[P]\right)  ^{2}\right]
  =\var\left[  \widehat{Y}\right]  +\left(  E\left[  \widehat{P}_{L}\right]
    -E[P]\right)  ^{2} \leq
  c_{2}^{\prime}N_{0}^{-1}+c_{2}T^{\beta} \sum_{l=1}^{L}N_{l}^{-1}
  M^{-l\beta}+c_{1}^{2}h_{L}^{2\alpha}. 
\end{equation*}
Now we need to minimize the total computational work
\begin{equation*}
  C \leq c_{3}N_{0}h_{0}^{-\gamma} + c_{3}\sum_{l=1}^{L}N_{l} h_l^{-\gamma}  =
  c_{3}T^{-\gamma} \left[  N_{0}+ \sum_{l=1}^{L}N_{l} M^{\gamma l}\right]
\end{equation*}
under the constraint $\operatorname{MSE}\leq\varepsilon^{2}$. We first assume
$L$ to be given and minimize over $N_{0},\ldots,N_{L}$, and then we try to
find an optimal $L$. We consider the Lagrange function
\begin{multline*}
f(N_{0},\ldots,N_{L},\lambda)\equiv c_{3}T^{-\gamma} \left[  N_{0} +
\sum_{l=1}^{L} N_{l} M^{\gamma l}\right]  +\\
+\lambda\left(  c_{2}^{\prime}N_{0}^{-1} + c_{2}T^{\beta} \sum_{l=1}^{L}
N_{l}^{-1}M^{-l\beta}+c_{1}^{2}h_{L}^{2\alpha}-\varepsilon^{2}\right)  .
\end{multline*}
Taking derivatives with respect to $N_{l}$, $0\leq l\leq L$, we arrive at
\begin{gather*}
\frac{\partial f}{\partial N_{0}} = c_{3}T^{-\gamma} - \lambda c_{2}^{\prime}
N_{0}^{-2}=0,\\ 
\frac{\partial f}{\partial N_{l}}=c_{3}T^{-\gamma} M^{\gamma l}-\lambda
c_{2}T^{\beta}  M^{-l\beta}N_{l}^{-2}=0,
\end{gather*}
implying that
\begin{subequations}
\label{eq:optimal_Nl}%
\begin{gather}
N_{0}=\sqrt{\lambda}\sqrt{\frac{c_{2}^{\prime}}{c_{3}}} T^{\gamma/2},\\
N_{l}=\sqrt{\lambda}\sqrt{\frac{c_{2}}{c_{3}}} T^{(\gamma+\beta)/2}
M^{-l(\gamma+\beta)/2},\ 1\leq l\leq L,
\end{gather}
which we insert into the bound for the MSE to obtain
the Lagrange multiplier
\end{subequations}
\begin{equation}
\sqrt{\lambda} = \left[  \sqrt{c_{2}^{\prime} c_{3}} T^{-\gamma/2} +
  \sqrt{c_{2} c_{3}} T^{-(\gamma-\beta)/2}
  \frac{M^{L(\gamma - \beta)/2}-1}{M^{(\gamma-\beta
)/2-1}} \right]  \cdot \left[\varepsilon^{2} - c_{1}^{2} T^{2\alpha
} M^{-2\alpha L}\right]^{-1}.\label{eq:multiplier}
\end{equation}
By construction, we see that for any such choice of $N_{0},\ldots,N_{L}$, the
MSE is, indeed, bounded by $\varepsilon^{2}$. For fixed $L$, the total
complexity is now given by
\begin{multline}
\label{CL}
C(L) 
  = \left[ \sqrt{c_2' c_3} T^{-\gamma/2} + \sqrt{c_2c_3}
    T^{-(\gamma-\beta)/2} M^{(\gamma-\beta)/2}
    \frac{M^{L(\gamma-\beta)/2} - 1}{M^{(\gamma-\beta)/2} - 1}\right]^2
  \frac{1}{\varepsilon^{2}-c_{1}^{2} T^{2\alpha}
     M^{-2\alpha L}}. 
\end{multline}
In general, the optimal (but real-valued) choice of $L$ would now be the
arg-min of complexity estimate corresponding to the above choices of $N_l$,
which we could not determine explicitly in an arbitrary regime.  

We now turn to the second step, i.e., to integer-valued parameter choices. We
parametrize the optimal choice of $L$ by $d_1$ in
\begin{equation}
  L=\left\lceil \frac{\log\left(  d_{1}c_{1}T^{\alpha}
        \varepsilon^{-1}\right)}{\alpha\log(M)}\right\rceil. \label{eq:Loptimal}
\end{equation}
The proper choice of the parameter $d_1$ is discussed in detail in the
subsequent subsection~\ref{sec:balanc-bias-vari}.  Moreover, we choose with
$\kappa=\frac{\gamma-\beta}{2\alpha}$
\begin{subequations}
\label{eq:Nloptimal}%
\begin{align}
N_{0}  &= \left\lceil \sqrt{c_{2}^{\prime}}
    \frac{d_1^2}{d_1^2 - 1} \left(  \sqrt{c_{2}^{\prime}} 
      + \sqrt{c_{2}} T^{\beta/2}
      M^{\alpha\kappa} \frac{d_{1}^{\kappa} c_{1}^{\kappa
} T^{\alpha\kappa} \varepsilon^{-\kappa} -
1}{M^{\alpha \kappa}-1}\right) \varepsilon^{-2}  \right\rceil
,\\
N_{l}  &= \left\lceil \sqrt{c_2}
  \frac{d_1^2}{d_1^2-1} \left( \sqrt{c_2'} T^{\beta/2} + \sqrt{c_2}
    T^{\beta} \frac{d_1^\kappa c_1^\kappa
      T^{\alpha\kappa} \varepsilon^{-\kappa} -
      1}{M^{\alpha\kappa} -1} \varepsilon^{-2} M^{-l(\beta+\gamma)/2} \right)\right\rceil ,
\end{align}
$1\leq l\leq L$.
\end{subequations}

By construction, the MSE will be bounded by $\varepsilon^{2}$
  using the choices \eqref{eq:Loptimal} and~\eqref{eq:Nloptimal}. In the next
  step, we insert these definitions into the complexity bound. In order to
  obtain suitable simplifications, we use that $\lceil x \rceil \le x + 1$ for
  real $x$.%  Introducing shorthand notations

After a tedious calculation, we finally arrive at the expression
\begin{equation}
C\leq c_{4}\varepsilon^{-2(1+\kappa)}+c_{5}\varepsilon^{-\gamma/\alpha}+c_{6}
\varepsilon^{-(2+\kappa)}+c_{7}\varepsilon^{-2}+c_{8}, \label{eq:complexity-bound}%
\end{equation}
where (once more including the dependence on $M_0$)
\begin{subequations}
\label{eq:ml-constants}
\begin{align}
  c_{4}  &= c_{1}^{2\kappa} c_{2} c_{3} \frac{d_{1}^{2+2\kappa}}{d_1^2-1}
  \frac{M^{3\alpha\kappa}}{\left( M^{\alpha\kappa} - 1\right)^2},\\
  c_{5} &= c_{1}^{\gamma/\alpha} c_{3} d_1^{\gamma/\alpha}
  \frac{M^\gamma}{M^\gamma - 1},\\ 
  c_{6} &= c_1^\kappa c_3 \frac{d_1^{2+\kappa}}{d_1^2-1} \biggl[ \sqrt{c_2 c_2'}
  T^{-\gamma/2} M_0^{\gamma/2}
  - 2 c_2 T^{-\kappa\alpha} M_0^{\kappa\alpha}
  \frac{M^{\alpha\kappa}}{M^{\alpha\kappa} -1} \biggr] 
  \frac{M^{\alpha\kappa}(1+M^{\alpha\kappa})}{M^{\alpha\kappa}-1},\\  
  c_{7}  &= c_3 \frac{d_1^2}{d_1^2-1} \biggl[c_2' T^{-\gamma} M_0^{\gamma} -
  \sqrt{c_2} (1+\sqrt{c_2'}) T^{-(\gamma-\beta/2)} M_0^{\gamma-\beta/2}
  \frac{M^{\alpha\kappa}}{M^{\alpha\kappa} - 1} \\ 
  &\quad  + c_2
  T^{-2\alpha\kappa} M_0^{2\alpha\kappa}
  \frac{M^{2\alpha\kappa}}{\left( M^{\alpha\kappa} - 1 \right)^2}
  \biggr],\nonumber{}\\ 
c_{8}  &= c_3 T^{-\gamma} M_0^\gamma.
\end{align}
\end{subequations}

If $2\alpha>\beta$, then $\varepsilon^{-2(1+\kappa)}$ is the
  leading order term, and we obtain the first alternative in the complexity
  bound of the theorem statement. If $2\alpha < \beta$, then
  $\varepsilon^{-\gamma/\kappa}$ is the leading order term---i.e., the total
  work is dominated by the work needed for simulating one trajectory at the
  finest level---, and we obtain the third alternative in the complexity
  bound. Finally, when $2\alpha = \beta$, then the two terms are of the same
  order.
\end{proof}

  \begin{remark}
    \label{rem:complexity-explosion}
    Note that the complexity bound~\eqref{eq:complexity-bound} needs to be
    considered with care when $\beta$ is close to $\gamma$. Indeed, the
    coefficients $c_4$ and $c_6$ explode as $\beta \to \gamma$. However,
    please note that for $\varepsilon$ fixed and $\beta \to \gamma$ we have
    that the corresponding powers $\varepsilon^{-2(1+\kappa)}$ and
    $\varepsilon^{-(2+\kappa)}$ converge to $\varepsilon^{-2}$ as $\kappa \to
    0$. Moreover, a closer look at \eqref{eq:ml-constants} verifies that $c_4
    \varepsilon^{-2(1+\kappa)} + c_6 \varepsilon^{-(2+\kappa)}$ remains
    bounded as $\beta \to \gamma$. Hence, there is no inconsistency in the
    results reported in Theorem~\ref{thr:ml-complexity} with the results of
    \cite{G08b}.
  \end{remark}

Next, we consider the situation when the actual complexity bound for computing
one trajectory has logarithmic terms in $h_l^{-1}$, for instance in the case
of fractional Brownian motion.
\begin{theorem}
  \label{thr:ml-complexity-log}
  In the setting of Theorem~\ref{thr:ml-complexity}, we replace the complexity
  bound (iv) by the new condition
  \begin{itemize}
  \item[(iv')] $C_l \le c_3 N_l h_l^{-\gamma} \log(h_l^{-1})$, $0 \le l$.
  \end{itemize}
  Then the choices of $L$ and $N_l$, $0 \le l \le L$, given in
  (\ref{eq:Loptimal}) and (\ref{eq:Nloptimal}), respectively, lead to a mean
  squared error bound $\operatorname{MSE} \le \varepsilon^2$ with a complexity
  bound 
  \begin{equation*}
    C\leq
\begin{cases}
  c_4' \varepsilon^{-\frac{\gamma+2\alpha-\beta}{\alpha}} \log(\varepsilon^{-1}) + o\left(
    \varepsilon^{-\frac{\gamma+2\alpha-\beta}{\alpha}} \log(\varepsilon^{-1}) \right), & 2 \alpha >
  \beta,\\ 
  (c_4'+c_5') \varepsilon^{-\frac{\gamma+2\alpha-\beta}{\alpha}} \log(\varepsilon^{-1}) + o\left(
    \varepsilon^{-\frac{\gamma+2\alpha-\beta}{\alpha}} \log(\varepsilon^{-1}) \right), & 2 \alpha =
  \beta,\\
  c_5' \varepsilon^{-\gamma/\alpha} \log(\varepsilon^{-1}) + o\left( \varepsilon^{-\gamma/\alpha}
 \log(\varepsilon^{-1}) \right), & 2 \alpha < \beta,
\end{cases}
  \end{equation*}
  where $c_i' = c_i / \alpha$, $i=4,5$.
\end{theorem}
\begin{proof}
  We choose $L$ and $N_0, \ldots, N_L$ as in
  Theorem~\ref{thr:ml-complexity}. As the mean squared error is not effected
  by the changed complexity bound, we obtain that the mean squared error is,
  once more, bounded by $\varepsilon^2$. On the other hand, for the complexity
  bound, we note that
  \begin{equation*}
    C \le c_3 \sum_{l=0}^L N_l h_l^{-\gamma} \log(h_l^{-1}) \le c_3
    \log(h_L^{-1}) \sum_{l=0}^L N_l h_l^{-\gamma}.
  \end{equation*}
  The last sum gives the upper bound given in~(\ref{eq:complexity-bound}).
  On the other hand, by~(\ref{eq:Loptimal}) (using once more $\lceil x \rceil
  \le x+1$ for $x \in \mathbb{R}$), we have
  \begin{equation*}
    \log(h_L^{-1}) \le \log\left( M_0 M d_1^{1/\alpha} c_1^{1/\alpha} \right)
    + \frac{\log(\varepsilon^{-1})}{\alpha}. \qedhere
  \end{equation*}
\end{proof}

\begin{table}[!htb]
  \centering
  \begin{tabular}{|l|c|c|c|}
    \hline
    & MLMC & classical MC & speed up of MLMC \\
    \hline
    Generic & $\varepsilon^{-(\gamma+2\alpha-\beta)/\alpha}$ & 
    $\varepsilon^{-(2+\gamma/\alpha)}$ & $\varepsilon^{-\beta/\alpha}$\\
    $\alpha = \beta/2$ & $\varepsilon^{-\gamma/\alpha}$ &
    $\varepsilon^{-(2+\gamma/\alpha)}$ & $\varepsilon^{-2}$ \\
    $\alpha = \beta$ & $\varepsilon^{-(\gamma/\alpha+1)}$ &
    $\varepsilon^{-(2+\gamma/\alpha)}$ & $\varepsilon^{-1}$ \\
    \hline
  \end{tabular}
  \caption{Comparison of asymptotic complexity for multilevel and standard
    Monte Carlo in the framework of Theorem~\ref{thr:ml-complexity}---i.e.,
    ignoring log-terms. $\alpha$
    denotes the weak order of convergence, $\beta/2$ the strong order. We
    distinguish the cases $\alpha = \beta/2$ and $\alpha = \beta$. We choose
    the work rate to be $\gamma = 1$, which is the most relevant case (up to
    logarithmic terms).}
  \label{tab:summary-general}
\end{table}

Under the assumptions of Theorem~\ref{thr:ml-complexity}, we can summarize the
complexity requirements for the multi-level and the classical Monte Carlo
methods, respectively, to obtain an MSE of order $\varepsilon^2$, see
Table~\ref{tab:summary-general}. In particular, note that the complexity of
classical Monte Carlo is asymptotically worse by a factor $\varepsilon^{-2}$ in
the ``non-regular'' case, when the weak rate is equal to the strong rate,
but still worse by a factor $\varepsilon^{-1}$ when the weak rate is actually
twice as good as the strong rate. With Theorem~\ref{thr:ml-complexity-log} we
get the same speed ups, as the logarithmic terms appear for both multilevel
and single level Monte Carlo.

\subsection{Balancing bias and variance in the multilevel algorithm}
\label{sec:balanc-bias-vari}

In the now classical works of Giles on multilevel Monte Carlo, the choice
$d_{1} = \sqrt{2}$ is advocated, see for instance~\cite%
{G08b}. This means that we reserve the same error tolerance $\varepsilon/2$
both for the bias or discretization error and for the statistical or Monte
Carlo error. Indeed, the choice of $d_{1}$ corresponds to the distribution of
the total MSE $\varepsilon^{2}$ between the statistical and the discretization
error according to
\begin{equation*}
\varepsilon^{2} =
\underbrace{\frac{\varepsilon^{2}}{d_{1}^{2}}}_{\text{disc.~error}}  +
\underbrace{\left( 1 - \frac{1}{d_{1}^{2}} \right)
  \varepsilon^{2}}_{\text{stat.~error}} . 
\end{equation*}

In many situations, the choice $d_1 = \sqrt{2}$ is not optimal. In fact, even
in an ordinary Monte Carlo framework, one should not blindly follow this
rule. For instance, for an SDE driven by a Brownian motion, the Euler scheme
usually (i.e., under suitable regularity conditions) exhibits weak
convergence with rate $1$. Assuming the same constants for the weak error
and the statistical error, a straightforward optimization will show that it
is optimal to choose the number of timesteps and the number of Monte Carlo
samples such that the (squared) discretization error is $\varepsilon^2/3$ and
the (squared) statistical error is $2\varepsilon^2/3$.

Let us now study in detail the asymptotic behavior for
$\varepsilon\downarrow0$ of the optimal choice for $d_{1}$ for a given fixed
$M$ in the context of Theorem~\ref{thr:ml-complexity-log}. Once again assuming
$M_0=1$ by absorbing it into $T$, we obtain the complexity
\[
D(L)=\left[  \sqrt{c_{2}^{\prime}c_{3}}T^{-\gamma/2}+\sqrt{c_{2}c_{3}%
}T^{(\beta-\gamma)/2}M^{(\gamma-\beta)/2}\frac{M^{L(\gamma-\beta)/2}%
-1}{M^{(\gamma-\beta)/2}-1}\right]  ^{2}\frac{L\log M-\log T}{\varepsilon
^{2}-c_{1}^{2}T^{2\alpha}M^{-2\alpha L}},
\]
i.e., \eqref{CL} with an extra factor stemming from $\log h_{L}^{-1}= L \log
M-\log T$. (We note that we disregard integer constraints on $L$ and $N_l$ for
this asymptotic analysis.) We assume in advance that $L$ is such that $L\log
M-\log T>0$.  Setting $\frac{d}{dL}\log(D(L)) = 0$ yields
\begin{multline*}
\underbrace{\frac{2\sqrt{c_2} T^{\beta/2} M^{(\gamma-\beta)/2} }{\sqrt{c_2'} + \sqrt{c_2} T^{\beta/2} M^{(\gamma-\beta
)/2} \frac{M^{L(\gamma-\beta)/2}-1}{M^{(\gamma-\beta)/2}-1}} 
\frac{M^{L(\gamma-\beta)/2}}{M^{(\gamma-\beta)/2}-1}
\frac{\gamma-\beta}{2}}_{\equiv h_M(L,\alpha,\beta,\gamma)} +
\frac{1}{L\log M-\log T} = \frac{2\alpha c_1^2 T^{2\alpha}}{\varepsilon^{2}
  M^{2\alpha L}- c_1^2T^{2\alpha}}.
\end{multline*}
Noting that $h_M>0$,
\[
\lim_{L\rightarrow\infty}h_{M}(L,\alpha,\beta,\gamma)=\gamma-\beta>0, \quad
h_{M}(0,\alpha,\beta,\gamma) = \frac{2\sqrt{c_2' c_3}T^{-(\gamma-\beta)/2}
  M^{(\gamma -\beta)/2}}{\sqrt{c_2'c_3} T^{-\gamma/2}\frac{2}{\gamma-\beta}
  (M^{(\gamma-\beta)/2}-1)}, 
\]
% and in particular,
% \[
% h_{M}(L,\alpha,\beta,\gamma)=\left(  \gamma-\beta\right)  \left(
% 1+O(e^{-\frac{\gamma-\beta}{2}L\log M})\right)  \text{ \ \ for }L\rightarrow
% \infty,
% \]
we may rewrite the above equality as
\begin{equation}
\frac{2\alpha c_1^2 T^{2\alpha}}{h_{M}(L,\alpha,\beta,\gamma)+\frac{1}{L\log
    M-\log T}} + c_1^2T^{2\alpha}=\varepsilon^{2}M^{2\alpha L}.\label{ginn}%
\end{equation}
Clearly, the l.h.s.~of (\ref{ginn}) is bounded from below by $c_1^2
T^{2\alpha}$ if 
$L>\log\left(  T/M\right).$ That is, if $\varepsilon\downarrow0$ we must have
$M^{2\alpha L}\rightarrow\infty,$ hence $L\rightarrow\infty.$ Then by taking
logarithms of (\ref{ginn}) and solving for $L$ we get for $\varepsilon
\downarrow0,$%
\[
L=\frac{\log\varepsilon^{-1}}{\alpha\log M}+\frac{1}{2\alpha\log M}\log\left(
\frac{2\alpha c_1^2 T^{2\alpha}}{h_{M}(L,\alpha,\beta,\gamma) + \frac{1}{L\log
    M-\log T}}+c_1^2 T^{2\alpha}\right).
\]
Then, for $\varepsilon\downarrow0,$ i.e. $L\rightarrow\infty,$ it follows that%
\begin{equation}\label{eq:L-opt-asym}
L^{\ast} 
=\frac{1}{\alpha\log M}\log\left(  \varepsilon^{-1}\left(  \frac{2\alpha
c_1^2T^{2\alpha}}{\gamma-\beta}+c_1^2T^{2\alpha}\right)  ^{1/2}\right)  +O(\frac
{1}{\log\varepsilon^{-1}}).
\end{equation}
Now, with $d_{1}$ implicitly defined by $L^\ast = \frac{\log\left( d_1 c_1
    T^\alpha \varepsilon^{-1} \right)}{\alpha \log M}$,
cf.~\eqref{eq:Loptimal}, we obtain
\begin{equation}\label{eq:d1-opt-asym}
d_{1}=\left(  \frac{2\alpha}{\gamma-\beta}+1\right)  ^{1/2}\left(
1+O(\frac{1}{\log\varepsilon^{-1}})\right).
\end{equation}
\begin{remark}
  Notice that the dominant term in the asymptotically optimal choice for $d_1$
  only depends on the rates $\alpha, \beta, \gamma$, but not on the constants.
\end{remark}
Plugging the optimal choice $L = L^\ast$ into the (approximate) computational
cost $D$, we obtain
\begin{equation*}
D(L^{\ast}) = \frac{2\alpha+\gamma-\beta}{2\alpha^{2}}\frac{4c_2c_3
  T^{-\gamma} M^{(\gamma-\beta
)}c_{1}^{\frac{\gamma-\beta}{\alpha}}}{2 \left( M^{(\gamma-\beta)/2} - 1
\right)^2} f\left(  \frac{\gamma-\beta}{2\alpha
}\right)  \varepsilon^{-\frac{2\alpha+\gamma-\beta}{\alpha}}\log\varepsilon
^{-1} \left( 1 + O\left(\frac{1}{\log \varepsilon^{-1}} \right) \right)
\end{equation*}
with $f(x):=\left(  1+\frac{1}{x}\right)  ^{x}$
being increasing with $f(0+)=1$ and $f(x)\rightarrow e$ when $x\rightarrow
\infty.$ Comparing with Giles' choice $d_{1}=\sqrt{2},$ i.e.,
\begin{equation*}
L^{\text{Giles}}=\frac{\log\left(  \sqrt{2}c_{1}T^{\alpha}\varepsilon
^{-1}\right)  }{\alpha\log M}
\end{equation*}
we get% %
\[
\frac{D\left( L^{\text{Giles}} \right)}{D\left( L^\ast \right)}%
=\frac{2^{1+\frac{\gamma-\beta}{2\alpha}}}{\left(  1+\frac{\gamma-\beta
}{2\alpha}\right)  f\left(  \frac{\gamma-\beta}{2\alpha}
\right)  } \left( 1 +
O\left(\frac{1}{\log \varepsilon^{-1}} \right) \right).
\]
Note that the above fraction---plotted in
Figure~\ref{fig:Dfraction}---(asymptotically) takes its minimum value of $1$
when $\frac{\gamma-\beta}{2\alpha} = 1$. In this case, our proposed
asymptotically optimal choice of $d_1$ takes the value $\sqrt{2}$ (up to
higher order terms) and, thus, coincides with Giles' choice.
\begin{figure}[!htp]
  \centering
  \includegraphics[width=0.5\textwidth]{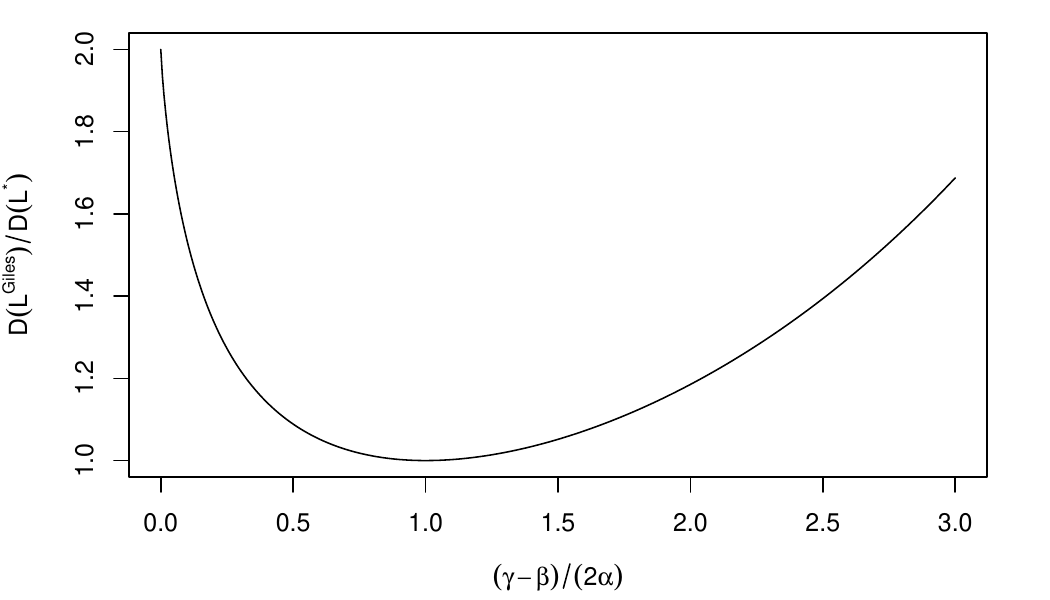}
  \caption{Dominant term of $\frac{D\left( L^{\text{Giles}} \right)}{D\left(
        L^\ast \right)}$ as a function of $\frac{\gamma-\beta}{2\alpha}$} 
  \label{fig:Dfraction}
\end{figure}

\begin{remark}
  The special case $\gamma = \beta$ requires new calculations (due to
  exploding terms). In this case, one can see that
  \[
  L^\ast = \frac{\log\varepsilon^{-2}}{2\alpha\log M}+\frac{\log
    \log\varepsilon^{-2}}{2\alpha\log M} - \frac{\log\log M^{2\alpha}}{2\alpha
  \log M} + \frac{1}{2} c_1^2 T^{2\alpha} + O\left( \frac{\log\log
    \varepsilon^{-2}}{\log \varepsilon^{-2}} \right)
  \]
  corresponding to
  \[
  d_{1} = \frac{M^{c_1^2 \alpha T^{2\alpha}/2}}{c_1 T^\alpha \sqrt{\alpha M}}
  \sqrt{\log \varepsilon^{-1}} \left( 1 + O\left( \frac{\log\log
        \varepsilon^{-2}}{\log \varepsilon^{-2}} \right) \right).
  \]
  We see that $d_1 \rightarrow\infty$ as $\varepsilon \downarrow
  0$. Nonetheless, the relative computational costs compared to $d_1 =
  \sqrt{2}$ stays bounded as
  \[
  \frac{D\left( L^{\text{Giles}} \right)}{D \left( L^\ast \right)} = 2 \left(
    1 + O\left( \frac{1}{\log \varepsilon^{-1}} \right) \right).
  \]
\end{remark}

Let us note an important difference between the multilevel algorithm for the
irregular case $\beta < 1$ explored here and the classical multilevel
algorithm of Giles~\cite{G08b} regarding the distribution of the
work-load. For the case of a classical SDE, the work is going to be
essentially equi-distributed among the levels. For the rough SDE case
considered here, we see from the proof of Theorem~\ref{thr:ml-complexity} that
most of the computational budget is actually spent on the fine grids, i.e., on
the levels with high index $l$. This is schematically represented in
Figure~\ref{fig:work-dist}, where we used the theoretical complexity estimates
from the proof of Theorem~\ref{thr:ml-complexity} with $\beta = 0.6$, $\alpha
= 0.3$ and the remaining constants set to arbitrary values.
\begin{figure}[!htb]
  \centering
  \includegraphics[width=0.5\textwidth]{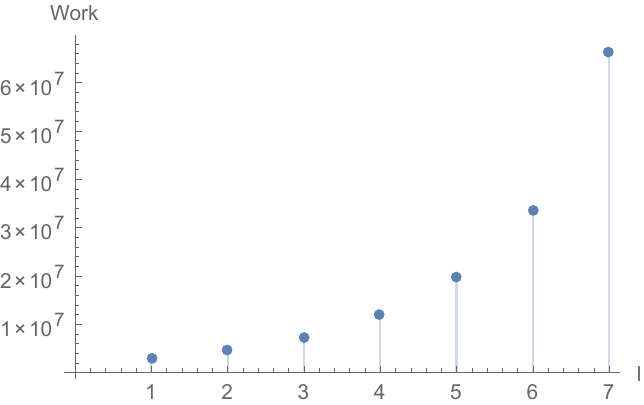}
  \caption{Work distribution among levels. Schematic presentation
      based on a case of $\beta = 0.6$.}
  \label{fig:work-dist}
\end{figure}

\subsection{Multilevel Monte Carlo for RDEs}

We shall now combine the results of Section~\ref{sec:prob_conv_for_rdes} and
Section~\ref{sec:mult-simul-rdes}. For convenience, we recall the regularity
assumptions of the convergence analysis (cf.~Condition~\ref{mixed_hoel}):
Let $X = (X^1,\ldots,X^d)\colon [0,T] \rightarrow \mathbb{R}^{d}$ be a centered, continuous Gaussian process with independent components. Assume that the covariance of every component has finite mixed $(1,\rho)$-variation for some $\rho \in [1,2)$ on $[0,T]^2$, that is, for $k=1,\dots,d$,
\begin{equation*}
\sup_{\substack{(t_{i}), (t_{j}^{\prime}) \in\mathcal{D}([0,T])
}}
\left(\sum_{t'_{j}}\left(\sum_{t_{i}}\left|    E \left[ X^k_{t_{i},t_{i+1}}X^k_{t_{j}^{\prime},t_{j+1}^{\prime}}   \right]   
\right|\right)^{\rho}\right)^{\frac{1}{\rho}} \\ 
< \infty .
\end{equation*}%

%Let $X\colon [0,T] \to \R^d$ be Gaussian process with for some $\rho \in
%[1,2)$. 
Consider the solution $Y \colon [0,T] \to \R^m$ of the RDE 
\begin{align*}
dY_t = V(Y_t)\,d\mathbf{X}_t; \quad Y_0 \in \R^m
\end{align*}
where $V = (V_1,\ldots,V_d)$ is a collection of vector fields in $\R^m$ with 
$|V|_{\text{Lip}^{\eta}} < \infty$ for some $\eta \geq \frac{2\rho}{\rho
- 1}$. Set $S := Y$ and let $S^{(h_l)}$ be the simplified step-3 Euler
approximation of $Y$ with mesh-size $h_l$ (in the case $\rho =1$, it
suffices to consider a step-2 approximation). Let $f \colon C([0,T],\R^m)
\to \R^n$ be a Lipschitz continuous functional and set $P := f(S)$, $%
\widehat{P}_l := f(S^{(h_l)})$. 

\begin{theorem}
  \label{thr:fBM-ml-comp}

Consider a functional of an RDE driven by Gaussian signal satisfying the above
assumptions, which is evaluated to within a MSE of $\varepsilon^{2}$.\\
(a) If we assume that the cost of sampling a vector of increments
of $X$ of length $N$ is proportional to $N\log(N)$, then an upper bound for
the complexity is given by
\begin{equation*}
\mathcal{O}\left( \varepsilon ^{-\theta }\right), \quad \forall \theta
>\frac{2\rho }{2-\rho}. 
\end{equation*}
(b) On the other hand, if we assume that the cost of obtaining such a sample
is proportional to $N^2$, then an upper bound of the total complexity is
\begin{equation*}
  \mathcal{O}\left( \varepsilon ^{-\theta }\right), \quad \forall \theta
>\frac{4\rho }{2-\rho}.
\end{equation*}
\end{theorem}
\begin{proof}
We want to calculate the quantities needed
in Theorem \ref{thr:ml-complexity}. As $f$ is assumed to be Lipschitz, the
weak rate of convergence is (at least) the strong rate of convergence, i.e.,
$\alpha =\beta /2$ in the notation of Theorem~\ref{thr:ml-complexity}. By
Corollary~\ref{cor:strong_rates_smooth_vf}, the strong rate is $\beta/2 =
\frac{2}{\rho} - \frac{1}{2} - \delta$ for any $\delta>0$.
Observe that
\begin{equation*}
\var\left[ \widehat{P}_{l}-P\right] \leq E\left[ \left( \widehat{P}%
_{l}-P\right) ^{2}\right] \leq \left\vert f\right\vert _{Lip}^{2}E\left[
\left\vert S^{\left( h_l \right) }-S \right\vert ^{2}\right] =O\left(
h_l^{\beta }\right)
\end{equation*}%
and 
\begin{equation*}
\var\left[ \widehat{P}_{l}-\widehat{P}_{l-1}\right] \leq \left( \var\left[ 
\widehat{P}_{l}-P\right] ^{1/2}+\var\left[ \widehat{P}_{l-1}-P\right]
^{1/2}\right) ^{2}=O\left( h_l^{\beta }\right)
\end{equation*}%
for all $\beta < \frac{2}{\rho} - 1$. Of course the variance of the average
of $N_{l}$ IID samples becomes%
\begin{equation*}
\var\left[ \widehat{Y}_{l}\right] =\frac{1}{N_{l}}\var\left[ \widehat{P}_{l}-%
\widehat{P}_{l-1}\right] =O\left( h_l^{\beta }/N_{l}\right) .
\end{equation*}%
This shows condition (iii) in Theorem~\ref{thr:ml-complexity}. Trivially, a
strong rate is also a weak rate, in the sense 
that%
\begin{equation*}
E\left( \widehat{P}_{l}-P\right) \leq E\left[ \left( \widehat{P}%
_{l}-P\right) ^{2}\right] ^{1/2}=O\left( h_l^{\beta /2}\right),
\end{equation*}%
which gives condition (i)
Condition (ii), ``unbiasedness'' is obvious for the estimator (\ref{defYhatl}%
). 

In case (a), condition (iv') of Theorem~\ref{thr:ml-complexity-log} holds, and
the theorem implies that an MSE $\varepsilon^2$ can be achieved at cost
proportional to
\begin{equation*}
  \varepsilon^{-\frac{1}{\alpha}} \log(\varepsilon^{-1}) = \varepsilon^{-\left(
      \frac{1}{\frac{1}{\rho} - \frac{1}{2}} + \delta \right)}
  \log(\varepsilon^{-1}) = \varepsilon^{-\left( \frac{2\rho}{2-\rho} + \delta
    \right)} \log(\varepsilon^{-1})
\end{equation*}
for any $\delta>0$. By choosing $\delta$ slightly larger, we may get rid of
the logarithmic term. In the end, we get $\mathcal{O}\left(
  \varepsilon^{-\theta} \right)$ for any $\theta > \frac{2\rho}{2-\rho}$.

On the other hand, in the general case (b), we rely on
Theorem~\ref{thr:ml-complexity} with $\gamma=2$. By similar calculations as
above (replacing $1/\alpha$ by $2/\alpha$), we arrive at the result for this
case. 
\end{proof}

  \begin{remark}
    If $X$ is a fractional Brownian motion with Hurst parameter $H>1/4$ (with
    $\rho = 1/(2H) < 2$), then we can generate samples of increments at cost
    proportional to $N \log(N)$ (e.g., by circulant embedding methods, see
    \cite{D04}), and we are, hence, in the situation of
    Theorem~\ref{thr:fBM-ml-comp} (a).
  \end{remark}

In Table~\ref{tab:complexity-summ-fbm} we compare typical asymptotic
complexities for RDEs driven by fractional Brownian motion for both the
multi-level and the classical Monte Carlo estimators. We distinguish between
the ``non-regular'' regime when $\alpha = \beta/2$ and the more favorable
regime when $\alpha = \beta$. Moreover, we have simplified the presentation in
Table~\ref{tab:complexity-summ-fbm} by neglecting the higher order
and logarithmic
terms. I.e., any complexity $\varepsilon^{-a}$ in
Table~\ref{tab:complexity-summ-fbm} should actually be understood as
$\varepsilon^{-a-\delta}$ for any $\delta>0$. Thus, when the Hurst parameter
is not too small, multi-level can make the difference between a feasible
simulation and a quite impossible one. E.g., when $H=2/5$ and the payoff
function $f$ is so irregular that the weak rate of convergence is not better
than the strong rate of convergence, the complexity for a standard Monte Carlo
estimator would be roughly of order $\varepsilon^{-5.33}$, whereas the
multi-level version would have complexity roughly of order
$\varepsilon^{-3.33}$, which is not much worse than the complexity of a
standard Monte Carlo estimator of the usual Brownian motion
regime. Admittedly, when $H=1/3$ and one has an irregular payoff, then both
standard and multi-level Monte Carlo are very costly computation wise.

\begin{table}[!htb]
  \centering
  \begin{tabular}{|l|c|c|c|c|}
    \hline
    & \multicolumn{2}{|c|}{$H = 2/5$} & \multicolumn{2}{|c|}{$H = 1/3$} \\
    \hline
    & MLMC & classical MC & MLMC & classical MC\\
    \hline
    $\alpha = \beta/2$ & $\varepsilon^{-10/3} \approx \varepsilon^{-3.33}$ &
    $\varepsilon^{-16/3} \approx \varepsilon^{-5.33}$ & $\varepsilon^{-6}$ &
    $\varepsilon^{-8}$ \\
    $\alpha = \beta$ & $\varepsilon^{-8/3} \approx \varepsilon^{-2.67}$ &
    $\varepsilon^{-11/3} \approx \varepsilon^{-3.67}$ & $\varepsilon^{-4}$ &
    $\varepsilon^{-5}$ \\
    \hline
  \end{tabular}
  \caption{Comparison of asymptotic complexities for multi-level and classical
    Monte Carlo for RDEs driven by fractional Brownian motion with Hurst index
    $H=2/5$ and $H=1/3$. We distinguish between the cases $\alpha = \beta/2$
    and $\alpha = \beta$. For this summary, we neglect the ``higher order
    terms'', i.e., we neglect the $\delta$ in $\beta/2 = 1/\rho-1/2-\delta$. For
    $H=2/5$, we set $\beta/2 = 3/10$, and for $H=1/3$ we set $\beta/2 = 1/6$.}
  \label{tab:complexity-summ-fbm}
\end{table}

\section{Numerical experiments}
\label{sec:numer-exper}

\subsection{A linear, non-smoothing example}
\label{sec:linear-non-smoothing}

\begin{figure}[!htp]
  \centering
  \includegraphics[width=\textwidth]{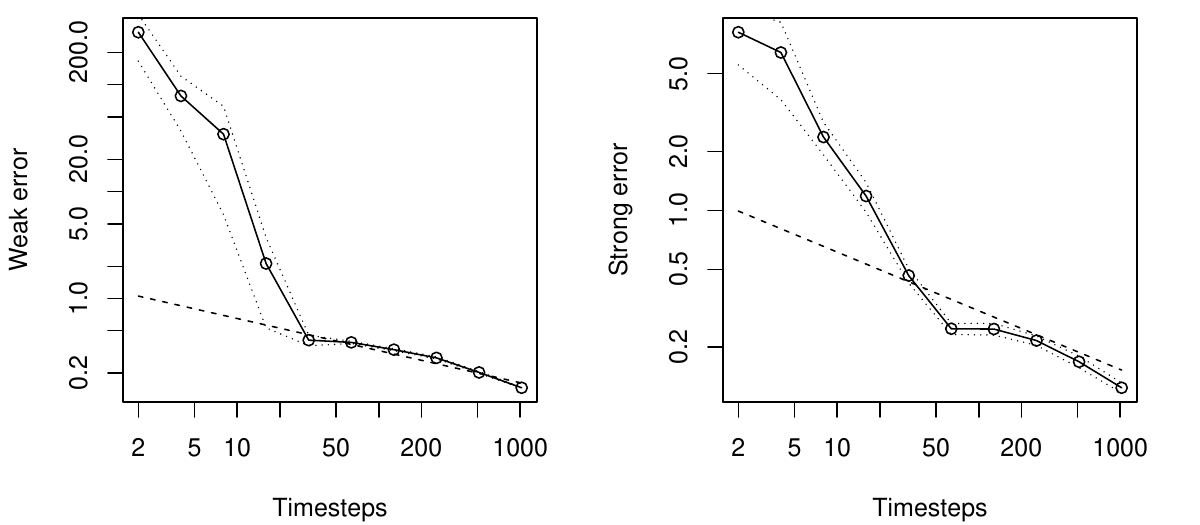}
  \caption{Strong and weak error for a fBm with Hurst index
    $H=0.4$. Dashed line corresponds to the theoretical strong rate of
    convergence $0.3$, dotted lines show confidence intervals around the error
  due to the integration error. Weak error corresponds to the functional $f(y)
  := (|y(1)| - 1)^+$.}
  \label{fig:err_s2_H04}
\end{figure}

We consider a linear RDE in $\mathbb{R}^3$ driven by a two-dimensional
fractional Brownian motion with Hurst index $H$. In fact, we consider vector
fields $V_i(y) = A_i y$, $y \in \mathbb{R}^3$, $i=1,\, 2$, with
\begin{equation*}
  A_1 =
  \begin{pmatrix}
    0 & 1 & 2 \\
    -1 & 0 & 1/2 \\
    -2 & -1/2 & 0
  \end{pmatrix},
  \quad A_2 =
  \begin{pmatrix}
    0 & 0.7 & 0.9\\
    -0.7 & 0 & 1\\
    -0.9 & -1 & 0
  \end{pmatrix}.
\end{equation*}

Note that the matrices $A_1$ and $A_2$ are anti-symmetric, implying that the
sphere $S^2$ is invariant under the solution of the SDE.  Note
  that the RDE driven by these vector fields is rather challenging from a
  numerical point of view, if we try to solve them using general,
  non-geometric schemes. In particular, in the case $H=1/2$, the equation is
  non-hypoelliptic and provides a smooth example in which the standard Euler
  scheme only has weak convergence rate $1/2$. In the case $H \neq 1/2$,
  we cannot expect any smoothing properties of the solution, either. Formally,
  further note that the vector fields are unbounded, violating one of our
  theoretical assumptions.

We implement the
simplified Euler scheme~\eqref{eq:simplified-stepN-euler}, where the
increments of the fractional Brownian motion were simulated by Hosking's
method, see~\cite{D04}.\footnote{The underlying Gaussian random numbers are
  simulated using the Box-M\"{u}ller method. The pseudo random numbers are
  generated by the Mersenne-Twister~\cite{MT98}.} Hosking's method is an exact
simulation method, i.e., if fed with truly Gaussian random numbers, it will
produce samples from the true distribution of increments of the fractional
Brownian motion. It is similar to the more obvious simulation method based on
the Cholesky factorization of the covariance matrix of the increments, but
preferable in terms of memory requirement, especially when grids of sizes of
up to $2^{14} = 16384$ are considered. As Cholesky's method, the complexity of
simulating the increments of the fractional Brownian motion on a grid with
size $M$ is essentially proportional to $M^2$, so we are working in the
context of Theorem~\ref{thr:fBM-ml-comp} (b). % There are non-exact methods
% with linear complexity, and even exact methods with linear complexity up to
% logarithmic terms, but we prefer Hosking's algorithm due to its simplicity.

Starting at $Y_0 = (1,0,0)$, Figure~\ref{fig:err_s2_H04} shows the strong and
weak convergence of the scheme for $H = 0.4$. More precisely, let $\barY_1^N$
denote the result of the scheme based on a uniform grid on $[0,1]$ based on
$N$ time-step. Then consider $\barY_1^{2N}$ based on the increments of the
\emph{same fBm}.\footnote{In practice, this means that we generated the
  increments of the fBm $X$ on the finer grid $\frac{k}{2N}$, $k=0, \ldots,
  2N$ and then obtained the increments on the coarser grid by adding the
  respective increments on the fine grid.} Then, the lower part of
Figure~\ref{fig:err_s2_H04} shows the Monte Carlo estimator of
$E\left[\left\lvert \barY^N_1 - \barY^{2N}_1 \right\rvert \right]$ plotted
against $N$. We, indeed, observe the expected rate of strong convergence,
which, due to Theorem \ref{cor:simplified-stepN-euler} is $2H - 1/2 = 0.3$,
but only after a prolonged pre-asymptotic phase.

In the upper panel of Figure~\ref{fig:err_s2_H04}, we plot the \emph{weak}
error for the calculation of $E\left[f\left( Y_1 \right) \right]$ for the
functional
\begin{equation*}
  f(y) := (|y| - 1)^+.
\end{equation*}
This implies that $E\left[f\left( Y_1 \right) \right] = 0$, so that we do not
need to carry out lengthy calculations in order to find an appropriately
accurate reference value. The figure indicates that the rate of the weak error
is again equal to the strong rate $0.3$. Note that the same would be true even
in the case $H=1/2$, because the Markov semigroup associated to the solution
(in the case $H=1/2$) is not smoothing and, in addition, the functional $f$ is
non-smooth on $S^2$, i.e., with probability $1$. Again, the roughness of the
driving signal leads to a remarkably strong pre-asymptotic regime. Indeed,
when the grid is too coarse, then the weak approximation error can be
huge. Visually, it seems that the asymptotic error analysis accurately
describes the true error when the mesh of the grid is at least around $0.02$
for the case $H = 0.4$.

\begin{figure}[!htp]
  \centering
  \includegraphics[width=\textwidth]{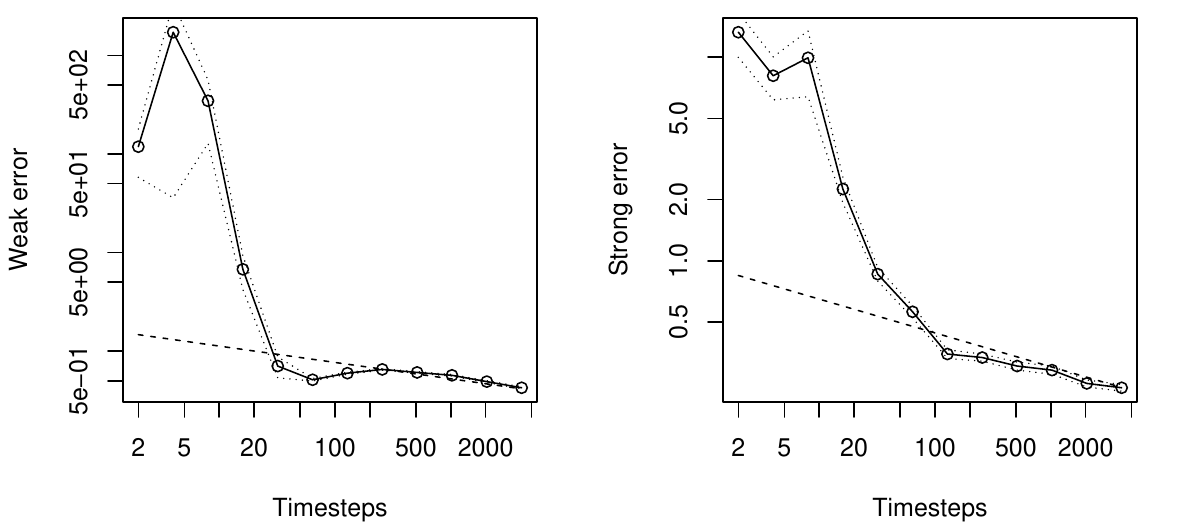}
  \caption{Strong and weak error for a fBm with Hurst index
    $H=0.33$. Dashed line corresponds to the theoretical strong rate of
    convergence $0.16$, dotted lines show confidence intervals around the error
  due to the integration error. Weak error corresponds to the functional $f(y)
  := (|y(1)| - 1)^+$.}
  \label{fig:err_s2_H033}
\end{figure}

Figure~\ref{fig:err_s2_H033} shows strong and weak errors for the same
differential equation and the same function $f$, but in the even rougher case
$H = 0.33$. In this case, the size of the errors for very coarse grids are
even larger than for $H=0.4$, and, moreover, the pre-asymptotic phase seems
even longer: here the mesh of the grid should probably be at least $0.01$ in
order to describe the true computational error by the asymptotic error
bounds.

\subsection{Multilevel Monte Carlo for the linear example}
\label{sec:mult-monte-carlo}

This long pre-asymptotic phase, in which the computational error is very
large, needs to be taken into account when constructing a successful
multi-level estimator: indeed, it is advisable to choose the coarsest grid
used in the multi-level iteration already within the asymptotic
regime. Thus, in the case $H=0.4$, we would recommend to choose $h_0 \le 0.02$
for this particular example. This is remarkably different from the standard
SDE case, where often $h_0$ is chosen to be equal to $T$, i.e., the coarsest
grid contains only the start and end points of the interval $[0,T]$. However,
when employing this strategy for the fBm example here, the constants in the
error bound for the multi-level estimator will completely overshadow the
asymptotic convergence rate, to the extent that even for long computation time
no ``empirical'' convergence is exhibited. Indeed, the coarsest levels then
combine a large error with an even larger variance, and this combination,
while harmless in the asymptotic limit $\varepsilon\to 0$, renders the standard
multi-level construction useless.

Fortunately, the picture is completely different when the coarsest grid is
chosen to be fine enough, in the current example for $H = 0.4$ this means $h_0
\le 0.02$. Then the multi-level algorithm requires considerably less
computational time for the same MSE tolerance than a classical MC estimator,
even for quite moderate levels of the tolerance. For this demonstration, we
choose a different function, namely
\begin{equation*}
  g(y) = |y| \mathbf{1}_{y^1>0}.
\end{equation*}
Indeed, the previously used function $f(y) = (|y| - 1)^+$ has the property
that $f(Y_1) \equiv 0$, so that the variance of $f(\overline{Y}_1^N)$ goes to
$0$ when $N \to \infty$. This, however, makes the basic idea of the multi-level
approach redundant, as the variance of estimators anyway decrease when the
mesh is decreased, even without the telescoping procedure. 

A direct comparison of the performance of the classical and the multi-level
Monte-Carlo estimator is difficult in our situation, as it is very hard to
obtain a reference value, i.e., a ``true result''. Moreover, by the same
reasoning the coefficients $c_i$ in Theorem~\ref{thr:ml-complexity} are very
difficult to estimate. Thus, we use the following
procedure to test the respective performances:
\begin{itemize}
\item Fix $L$, the number of levels in the multi-level procedure, and $h_0$,
  the coarsest grid. Here, we choose $h_0 = 1/64$ and $L = 7$. Thus, the
  finest grid in the multi-level Monte Carlo corresponds to $h_L = 0.00012 =
  1/8192$. As the fixed $L$ is probably sub-optimal, this choice is
  disadvantageous to the multi-level algorithm. We also choose the
  multiplication factor $M = 2$ here, and we parametrize the number of paths
  $N_l$ for the level $l$ by the number of paths $N_0$ at the coarsest level
  by some heuristic. In Table~\ref{tab:comparison-ml-clas}, we choose $N_0 =
  100$. Again, these non-optimal choices favour the classical
  Monte Carlo estimator.
\item Choose the mesh of the classical Monte Carlo estimator to be equal to
  $h_L$, the finest grid in the multi-level hierarchy. This guarantees that
  both estimators have the same bias -- even though we cannot easily estimate
  this bias due to the absence of a reference value.
\item Choose the number of paths in the classical Monte Carlo estimator and
  the number of paths in the coarsest grid for the multi-level estimator such
  that the complexity for the classical Monte Carlo estimator is equal to the
  complexity of the multi-level Monte Carlo estimator. We use an a-priori
  estimate for the complexity.
  \begin{itemize}
  \item For the classical Monte Carlo method, the complexity is estimated by
    the number of trajectories multiplied by the size of the grid.
  \item For the multi-level Monte Carlo method, the complexity at a level $l$
    is estimated by the product of the size of the finer grid and the number
    of trajectories for the level. The overall complexity is estimated by the
    sum of these complexity estimates for the individual levels.
  \end{itemize}
  Note that in practice, this complexity estimate is only given up to a
  constant of proportionality, which can be checked by comparing run-times on
  a computer.
\item Compute the sample variance for both estimators. If the sample variance
  for the multi-level Monte Carlo estimator is (significantly) smaller than
  the sample variance for the classical Monte Carlo estimator, then we,
  indeed, have demonstrated that the multi-level estimator will have a smaller
  MSE than the classical Monte Carlo estimator given the same computational
  budget, i.e., the same complexity.
\end{itemize}
The nice aspect of this procedure is that it allows a reliable comparison of
MSE given a certain complexity, even when the true MSE is not known because of
the absence of a reference value. However, we stress again that the
multi-level estimator constructed above will certainly not be optimal.In order
to take care of the constant in the complexity bound, we also compare
the actual run-times as empirical complexity estimates.

\begin{table}[!htb]
  \centering
  \begin{tabular}{|l|c|c|}
    \hline
    & Multilevel & Classical MC \\
    \hline
    Variance & $1.47 \times 10^{-2}$ & $1.90 \times 10^{-2}$ \\
    Time & $0.99$ s & $3.68$ s \\
    \hline
  \end{tabular}
  \caption{Variance and run-times for the multi-level and the classical Monte
    Carlo algorithm for fixed complexity and bias. Calculations are normalized
    by $N_0 = 100$.}
  \label{tab:comparison-ml-clas}
\end{table}

Table~\ref{tab:comparison-ml-clas} finds that for comparable complexity the
variance associated to the classical Monte Carlo estimator is considerably
lower than the variance of the classical Monte Carlo estimator. It is
interesting to note that the classical Monte Carlo estimator takes
considerably longer computational time. The reason is that the multi-level
algorithm uses the Euler scheme on coarser grids on average than the classical
Monte Carlo algorithm. As the complexity for sampling the increments of the
fractional Brownian motion increases quadratically in the size of the grid
when Hosking's method is applied, this explains why the computational time is
almost four times larger for the classical Monte Carlo method. Note that there
are other exact simulation methods with a complexity of order $\mathcal{O}(M
\log(M))$ in the grid size $M$, and approximate simulation methods even with
order $\mathcal{O}(M)$, see \cite{D04}. However, at least for the present,
linear differential equation, the simulation of the increments of the fBm will
always dominate the Euler steps, even when the complexity does only increase
linearly. Thus, the conclusions of Table~\ref{tab:comparison-ml-clas} should
hold irrespective of the simulation method.\footnote{The following heuristic
  calculation also supports this conclusion: assuming that we replace
  Hosking's algorithm by an algorithm with linear complexity and the same
  constant. Then we can easily predict the run-time of the classical Monte
  Carlo algorithm by dividing the run-time reported in
  Table~\ref{tab:comparison-ml-clas} by the size of the (finest) grid, i.e.,
  by $8192$, which gives a predicted run-time of $0.00045$ seconds. For the
  multilevel Monte Carlo method, the corresponding factor would be (with $M_l
  = T h_l^{-1}$ and $N_l = N_0 2^{-l(1+\beta)/2} = N_0 2^{-0.8l}$)
  \begin{equation*}
    \frac{M_0^2 N_0 + \cdots + M_L^2 N_L}{M_0 N_0 + \cdots + M_L N_L} = 2799,
  \end{equation*}
  giving a predicted run-time of $0.00035$ seconds, which is still lower then
  the predicted run-time for the classical Monte Carlo algorithm.
}

\subsection{A fractional Heston model}
\label{sec:fract-hest-model}

As a third example, let us consider an application from
  finance. One of the most popular asset price models is the Heston model, a
  stochastic volatility model, meaning that the diffusion coefficient of the
  asset price is itself stochastic. Recently, it has emerged that the
  stochastic volatility component is rougher than a standard Brownian motion
  (or a diffusion process) and should be modelled by a fractional Brownian
  motion or a process driven by a fractional Brownian motion, respectively,
  see Gatheral, Jaisson and Rosenbaum \cite{GJR14a} and Bayer, Friz and
  Gatheral~\cite{BFG15} and the references therein. Hence, the following
  \emph{fractional Heston model}---corresponding to the classical Heston model
  in Stratonovich formulation for $h = 1/2$---, see also Guennoun, Jacquier
  and Roome~\cite{GJR14b}, could be considered:
\begin{subequations}
  \label{eq:fHeston}
  \begin{align}
    dS_t &= -\frac{1}{2}\left(v_t + \frac{1}{2} \xi \varrho\right)S_t dt +
    \sqrt{v_t} S_t dW^1_t\\
    dv_t &= \left( \kappa (\theta - v_t) - \frac{1}{4} \xi^2 \right) dt + \xi
    \sqrt{v_t} ( \varrho dW^1_t + \sqrt{1-\varrho^2} dW^2_t ), 
  \end{align}
\end{subequations}
for a standard Brownian motion $W^1$ and an independent fractional Brownian
motion $W^2$ with Hurst index $1/4 < H < 1/2$. That is, in terms of the
RDE~\eqref{SDEIntro}, we choose the Gaussian process $X_t = (W^1_t, W^2_t)$,
which satisfies our assumptions with $\rho = 1/(2H)$. Note that the driving
noise of the price and the variance processes can be correlated using $-1 \le
\varrho \le 1$.

Notice that the fractional Heston model does not satisfy our regularity
assumptions in several ways:
\begin{itemize}
\item the vector fields are unbounded;
\item the vector fields are not differentiable at $v=0$ and not even defined
  for $v<0$.
\end{itemize}
Hence, depending on the choice of parameters and initial values $(S_0, v_0)$,
we may expect the rate to deteriorate. We also need to adjust the scheme in
order to preserve positivity of $(S,v)$, in this case by simply taking the
positive part after each time step.

On the other hand, one may expect the impact of the Brownian motion $W^1$ to
be stronger than the impact of the fractional Brownian motion $W^2$,
especially when considering standard payoff functions depending on the asset
price component $S$ only. Thus, the actual error based on a step size $h$
might look like $h^{-1/2}$ and $h^{-1}$ in the strong and weak sense,
respectively, when the parameters and initial values are ``nice enough'' and
$h$ is not too small.

\begin{figure}[!htp]
  \centering
  \includegraphics[width=\textwidth]{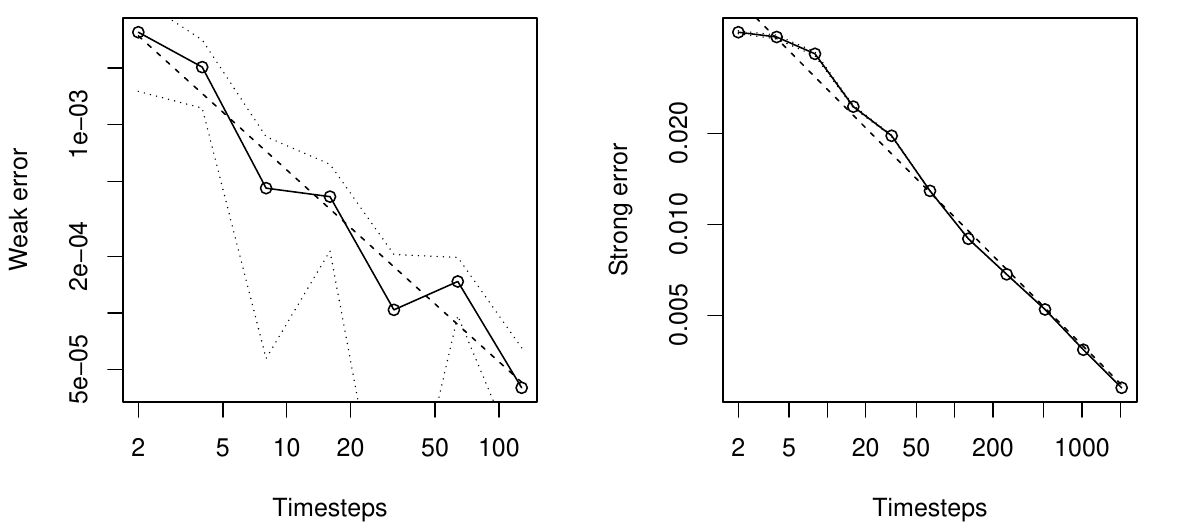}
  \caption{Weak and strong error for the fractional Heston model. Solid lines
    show the empirical errors, dashed lines show regression lines with rates
    $1.02$ in the weak and $0.42$ in the strong case. Dotted lines show
    confidence intervals around the solid lines.}
  \label{fig:fHeston}
\end{figure}
We choose model parameters $\kappa = 1$, $\theta = 0.16$, $\xi = 0.2$, $\rho =
-0.1$, $H = 0.4$. Notice that these parameters easily verify the Feller
condition $2\kappa \theta > \xi^2$ (which would imply $v_t > 0$ for $H =
1/2$). We further choose $S_0 = 1$, $v_0 = \theta$ and $r = 0$ and consider a
European call option with strike price $K = 1$.

As expected and shown in Figure~\ref{fig:fHeston}, we empirically observe
first order weak convergence and strong convergence with rate
$1/2$---actually, regression gives rates $1.02$ and $0.42$, respectively. This
is significantly better than the theoretical (strong) rate $0.3$. We expect
the rate to deteriorate eventually, when the number of timesteps is increased
even further.

We have also implemented the multilevel Monte Carlo algorithm for the
fractional Heston model (using the same parameters as above). We normalize the
workload to (roughly) $2\,610\,000$ individual Euler steps---i.e., we set the
cost of one Euler step to one. We choose $M_0 = M = 2$ and fix the bias by
requiring $h_L = \frac{1}{16}.$ Under these normalizations, the single level and
multilevel algorithms produce the variances and run times reported in
Table~\ref{tab:comparison-ml-clas-fHeston}.
\begin{table}[!htb]
  \centering
  \begin{tabular}{|l|c|c|}
    \hline
    & Multilevel & Classical MC \\
    \hline
    Variance & $1.9 \times 10^{-7}$ & $7.4 \times 10^{-7}$ \\
    Time & $0.88$ s & $0.79$ s \\
    \hline
  \end{tabular}
  \caption{Variance and run-times for the multi-level and the classical Monte
    Carlo algorithm for fixed complexity and bias.}
  \label{tab:comparison-ml-clas-fHeston}
\end{table}
As expected, the multilevel algorithm again clearly outperforms the single
level algorithm in producing an estimate with essentially $4$ times smaller
variance. We note that the parameters of the multilevel algorithm were based
on the \emph{observed} weak and strong rates of convergence, not the
theoretical ones.

\appendix

% \section{Precise computations ...}

\section{Elements of rough path theory}

We will now very briefly recall the elements of rough paths theory used in
this paper. For more details we refer to \cite{FV10}, \cite{LCL07}, \cite{LQ02} or \cite{FH14}. Our notation coincides with the one used in  \cite{FV10}.

Let $T^N(\R^d) =\R %
\oplus \R^d \oplus (\R^d \otimes \R^d) \oplus \ldots \oplus (\R^d)^{\otimes
N}$ be the truncated step-$N$ tensor algebra. We are concerned with $T^N(\R^d)$-valued
paths, as naturally given by iterated integrations of $\R^d$-valued smooth paths (``lifted smooth paths").
Such a path $\mathbf{x}$ has natural increments $\mathbf{x}_{s,t}\equiv \mathbf{x}_{s}^{-1}\otimes \mathbf{x}_{t}$.
The projection of such a path $\mathbf{x}$ on the first level is an $\R^d$-valued path and will be denoted by $\pi_1(%
\mathbf{x})$, the projection to $k$th level is denoted by $\pi_k$. Lifted smooth paths actually take values in $G^N(\R^d) \subset T^N(\R^d)$, where $G^N(\R^d)$ denotes the
free step-$N$ nilpotent Lie group with $d$ generators. The (left-invariant) Carnot-Caratheodory metric turns $(G^N(\R^d),d)$ into a metric space.

This already allows (see e.g. \cite{FV10}) to introduce the most commonly used  (homogenous) rough path ``norms"

\begin{eqnarray*}
\left\Vert \mathbf{x}\right\Vert _{p\text{-var;}\left[ 0,T\right] }
&=&\sup_{\left( t_{i}\right) \subset \left[ 0,T\right] }\left(
\sum_{i}d\left( \mathbf{x}_{t_{i}},\mathbf{x}_{t_{i+1}}\right) ^{p}\right)
^{1/p},  \label{DefinitionDoubleBarPvarOnGroup} \\
\left\Vert \mathbf{x}\right\Vert _{1/p\text{-H\"{o}l;}\left[ 0,T\right]
}  &=&   \sup_{0\leq s<t\leq T}\frac{d\left( \mathbf{x}_{s},\mathbf{x}_{t}\right) }{%
\left\vert t-s\right\vert ^{1/p}},
  \notag
\end{eqnarray*}%

and distances

\begin{eqnarray*}
d_{p%
\text{-var;}\left[ 0,T\right] }(\mathbf{x},\mathbf{y}) &= & \left(
\sup_{(t_i) \subset [0,T] }\sum_{i}d\left( \mathbf{x}_{t_{i},t_{i+1}},\mathbf{y}%
_{t_{i},t_{i+1}}\right) ^{p}\right) ^{1/p}, \\
d_{1/p\text{-H\"{o}l;}\left[ 0,T\right] }(\mathbf{x},\mathbf{y})
 &= & \sup_{0\leq s<t\leq T}\frac{d\left( \mathbf{x}_{s,t},\mathbf{y}%
_{s,t}\right) }{\left\vert t-s\right\vert ^{1/p}}
  \notag
\end{eqnarray*}%

where $p \in [1,\infty)$. Define also ``inhomogenous" variation and H\"{o}lder distances as follows.
For $k=$ $1,\ldots ,N,$%
\begin{equation*}
\rho _{p%
\text{-var;}\left[ 0,T\right] }^{\left( k\right) }\left( \mathbf{x},\mathbf{y%
}\right) =\sup_{\left( t_{i}\right) \subset \left[ 0,T\right] }\left(
\sum_{i}\left\vert \pi _{k}\left( \mathbf{x}_{t_{i},t_{i+1}}-\mathbf{y}%
_{t_{i},t_{i+1}}\right) \right\vert ^{p/k}\right) ^{k/p},
\end{equation*}%

and%
\begin{equation*}
\rho _{p\text{-var;}\left[ 0,T\right] }\left( \mathbf{x},\mathbf{y}\right)
=\max_{k=1,\ldots ,N}\rho _{p\text{-var;}\left[ 0,T\right] }^{\left(
k\right) }\left( \mathbf{x},\mathbf{y}\right) .
\end{equation*}%
\newline
Similarly, 

\begin{equation*}
\rho _{1/p\text{-H\"{o}l;}\left[ 0,T\right] }^{\left( k\right) }\left( 
\mathbf{x},\mathbf{y}\right) =\sup_{0\leq s<t\leq T}\frac{\left\vert \pi
_{k}\left( \mathbf{x}_{s,t}-\mathbf{y}_{s,t}\right) \right\vert }{\left\vert
t-s\right\vert ^{k/p}},
\end{equation*}%
\newline
and%
\begin{equation*}
\rho _{1/p\text{-H\"{o}l;}\left[ 0,T\right] }\left( \mathbf{x},\mathbf{y}\right)
 =\max_{k=1,\ldots ,N}\rho _{1/p\text{-H\"{o}l;}\left[ 0,T\right]
}^{\left( k\right) }\left( \mathbf{x},\mathbf{y}\right) .
\end{equation*}

Recall that a control function $\omega $ is a continuous function from  $\{ \, 0 \leq s \leq t \leq T\}$ to $[0,\infty)$, for which $\omega(s,t) + \omega(t,u) \leq \omega(s,u)$ holds for every $s \leq t \leq u$.
Note that $\omega(s,t)^{1/p}$ is a natural generalization of the quantity $|t-s|^{1/p}$ which appeared in the definition of all ``H\"older objects"

$$ \left\Vert \mathbf{x}\right\Vert _{1/p\text{-H\"{o}l;}\left[ 0,T\right]}, \,\,  %
    d_{1/p\text{-H\"{o}l;}\left[ 0,T\right] }(\mathbf{x},\mathbf{y}), \,\,  %
    \rho^k _{1/p\text{-H\"{o}l;}\left[ 0,T\right] }\left( \mathbf{x},\mathbf{y}\right), \,\,  %
    \rho _{1/p\text{-H\"{o}l;}\left[ 0,T\right] }\left( \mathbf{x},\mathbf{y}\right), \,\,  %
 $$
 defined above. Replacing $|t-s|^{1/p}$ by $\omega(s,t)^{1/p}$ then gives rise to similar norms and distances, denoted by 
$$ \left\Vert \mathbf{x}\right\Vert _{p\text{-}\omega;\left[ 0,T\right]}, \,\,  %
    d_{p\text{-}\omega;\left[ 0,T\right] }(\mathbf{x},\mathbf{y}), \,\,  %
    \rho^k _{p\text{-}\omega;\left[ 0,T\right] }\left( \mathbf{x},\mathbf{y}\right), \,\,  %
    \rho _{p\text{-}\omega;\left[ 0,T\right] }\left( \mathbf{x},\mathbf{y}\right). \,\,  %
 $$

By definition, a geometric $1/p$-H\"older rough path $\mathbf{x}$ is a path in 
$T^{\lfloor p \rfloor}(\R^d)$ which can be approximated by lifts of
smooth paths in the $d_{1/p-\text{H\"ol}}$ (equivalently:  $\rho_{1/p\text{-H\"{o}l}}$) metric; geometric $p$-rough
paths are defined similarly (with respect to the variation distance). Necessarily then, any such $\mathbf{x}$ takes values in
$G^{\lfloor p \rfloor}(\R^d) \subset  T^{\lfloor p \rfloor}(\R^d)$. The resulting rough path spaces are know to be Polish and are denoted by

\begin{align*}
C^{0,1/p-\text{H\"ol}}_0([0,T],G^{\lfloor p \rfloor}(\R^d))\quad \text{%
and} \quad C^{0,p-\text{var}}_0([0,T],G^{\lfloor p \rfloor}(\R^d)).
\end{align*}

If $V = (V_i)_{i=1,\ldots,d}$ is a collection of $\operatorname{Lip}^{\gamma}(\mathbb{R}^e)$ vector fields (in the sense of Stein, cf. \cite{FV10}) for some $\gamma > p$ and $\mathbf{x}$ is a geometric $p$-rough path, one can make sense of a unique solution $y\colon [0,T] \to \mathbb{R}^e$ of the equation
\begin{align*}
	dy_t = V(y_t)\,d\mathbf{x}_t; \quad y_0 \in \mathbb{R}^e
\end{align*}
and the solution depends (locally Lipschitz) continuously on the driving signal in the inhomogenous rough paths metric.

\subsection*{Acknowledgements}
\label{sec:acknowledgements}

P.F. has received funding from the European Research Council under the
European Union's Seventh Framework Program (FP7/2007-2013) / ERC grant
agreement nr. 258237. S.R. was supported by a scholarship from the Berlin
Mathematical School (BMS). C.~B.~and P.~F.~ acknowledge funding by the DFG
grants BA5484/1 and FR2943/2). All authors acknowledge support from the DFG
within Research Unit FOR 2402.

 This paper contains
results of S. Riedel's Ph.D. dissertation \cite[Ch.~6]{Rie13}, which was
written in collaboration with CB, PKF and JS.

\bibliographystyle{amsalpha.bst}
\bibliography{refs}

\end{document}